\documentclass[12pt,a4paper]{amsart}
\usepackage{graphicx}
\usepackage{amssymb}

 \usepackage[
  backend=bibtex,
  style=numeric,
    doi=false,
  sorting=nyt,
  maxbibnames=99,
  url=false,
    eprint=true,
    giveninits=true   
]{biblatex}

\renewbibmacro{in:}{\ifentrytype{article}{}{\printtext{\bibstring{in}\intitlepunct}}}
\AtBeginBibliography{\setlength{\emergencystretch}{3em}}

\usepackage{amsmath}
\usepackage{mathtools}
\mathtoolsset{showonlyrefs}
\usepackage[normalem]{ulem}
 \usepackage{soulutf8}
 \usepackage{cancel}

\DeclareMathOperator*\lowlim{\underline{lim}}
\DeclareMathOperator*\uplim{\overline{lim}}

\usepackage{amsthm}
\usepackage{a4wide}
\usepackage{pdfsync}
\usepackage{hyperref}
\usepackage{color}
\numberwithin{equation}{section}

\newtheorem{theorem}{Theorem}[section]
\newtheorem{proposition}[theorem]{Proposition}
\newtheorem{cor}[theorem]{Corollary}
\newtheorem{lemma}[theorem]{Lemma}

\theoremstyle{remark}
\newtheorem{remark}[theorem]{Remark}

\theoremstyle{definition}

\newcommand{\sgn}{\,{\rm sgn}}

\def\eps{\varepsilon}

\def\rr{\mathbb{R}}

\def\th{\mathcal{T}_h}

\def\eps{\varepsilon}

\title[]{Optimal  convergence rates for the finite element approximation of the Sobolev constant
}

\author[L. I. Ignat]{Liviu I. Ignat}
\address{
    [1] Institute of Mathematics ``Simion Stoilow'' of the Romanian Academy, 21 Calea Grivitei Street, 010702 Bucharest, Romania.
    \newline\indent
[2] Department of Mathematical Methods and Models, National University of Science and Technology Politehnica Bucharest, 313 Splaiul Independen\c tei, 060042 Bucharest, Romania.
    \newline\indent
[3] Academy of
Romanian Scientists, Ilfov Street, no. 3, Bucharest, Romania.}
\email[Corresponding author]{liviu.ignat@gmail.com}

\author[E. Zuazua]{Enrique Zuazua}
\address{
    [1] Friedrich-Alexander-Universit\"at Erlangen-N\"urnberg, 
    Department of Mathematics, Chair for Dynamics, Control, Machine Learning and Numerics (Alexander von Humboldt Professorship), 
    Cauerstr. 11, 91058 Erlangen, Germany.
    \newline\indent
    [2] Chair of Computational Mathematics,
 Deusto University, 48007 Bilbao, Basque Country, Spain.
    \newline\indent
    [3] Universidad Aut\'onoma de Madrid,
    Departamento de Matem\'aticas, 
    Ciudad Universitaria de Cantoblanco, 28049 Madrid, Spain.}
\email{enrique.zuazua@fau.de}

\date{}

\begin{document}

\begin{abstract}
We establish optimal convergence rates for the continuous piecewise affine finite element approximation of the Sobolev constant in arbitrary dimensions $N \geq 2$ and for Lebesgue exponents $1 < p < N$. Our analysis relies on a refined study of the Sobolev deficit in suitable quasi-norms, which have been introduced and utilized in the context of finite element approximations of the $p$-Laplacian. The proof further involves sharp estimates for the finite element approximation of Sobolev minimizers.
\end{abstract}


\maketitle
 \begin{center}
  \textit{Communicated by Endre S\" uli}\\
  \textit{Corresponding author: Liviu I. Ignat, liviu.ignat@gmail.com}
  \end{center}

\section{Introduction}

This paper is devoted to the study of the $P1$ finite element approximation of the Sobolev constant
\begin{equation}
\label{sobolev.min}
  S(p,N)=\inf_{u\in \dot W^{1,p}(\rr^N)\setminus\{0\}} \frac{\|D u\|_{L^p(\rr^N)}}{\| u\|_{L^{p^*}(\rr^N)}},
\end{equation}
which ensures the validity of the optimal Sobolev inequality
\begin{equation}\label{sobolev.in}
 \|D u\|_{L^p(\rr^N)}\ge S(p,N) \| u\|_{L^{p^*}(\rr^N)}, \, \forall u\in \dot W^{1,p}(\rr^N),
\end{equation}
in dimensions $N \geq 2$ and for exponents $1 \leq p < N$, where
\begin{equation}
p^* = \frac{Np}{N - p}
\end{equation}
denotes the Sobolev conjugate of $p$.

When $1<p<N$, the minimum constant $S(p,N)$ is attained on an $(N+2)$-dimensional manifold $\mathcal{M}$; see \cite{MR0448404, MR0463908}. The case $p = 1$, which involves distinct features and will not be addressed in this work, was treated in \cite{MR123260, MR126152}, where the sharp constant was determined.

The Sobolev constant plays an important role in various areas, including, of course, the existence and regularity of solutions for nonlinear PDEs. 
As already emphasized in \cite{MR2739020}, the determination of optimal constants in Sobolev inequalities, both in their discrete and continuous formulations, has significant analytical and computational implications. These constants allow for sharper error estimates in the numerical approximation of partial differential equations \cite{MR1115237}, facilitate the characterization of convergence domains for parameter-dependent iterative schemes employed in their resolution \cite{MR635927} and provide precise time scales governing the asymptotic behavior of solutions to time-dependent problems. Further applications and related developments can be found in \cite{MR2739020}.

For numerical analysis, we restrict our attention to bounded domains. To simplify the presentation, we consider the unit ball $B \subset \mathbb{R}^N$ with $N \geq 2$, equipped with a finite element mesh of characteristic size $h$. Let $V_h$ denote the corresponding finite-dimensional subspace of $W_0^{1,p}(B)$, consisting of $P1$ finite element functions on $B$ that are continuous, piecewise linear, and vanish on the boundary. These functions are extended by zero outside of $B$. A precise construction of $V_h$ is provided in Section~\ref{finite.elements}.

Sobolev's inequality still holds in $W_0^{1,p}(B)$, with the same sharp constant (see, e.g., \cite{MR2754218}),  but, in this case, the infimum is well known not to be achieved:
\begin{equation}
\label{sobolev.minb}
  S(p,N)=\inf_{u\in  W^{1,p}_0(B)\setminus\{0\}} \frac{\|D u\|_{L^p(B)}}{\| u\|_{L^{p^*}(B)}}.
\end{equation}
Clearly, the same inequality holds in the finite element subspace $V_h$ as well, with the following, slightly  larger, minimal constant 
\begin{equation}
S_h(p,N) :=\min _{u_h\in   V_h\setminus\{0\}}\frac{\|Du_h\|_{L^{p}(\rr^N)}} {\|u_h\|_{L^{p^*}(\rr^N)} }.
\end{equation}
Obviously,
\begin{equation}
S_h(p,N) \geq S(p,N).
\end{equation}
In view of the convergence properties of finite element methods, one can also prove that 
\begin{equation}
S_h(p,N) \to S(p,N), \, \hbox{as} \, h \to 0.
\end{equation}
Our goal here is to provide sharp convergence rates. This is relevant, of course, for  obtaining  sharp finite element convergence rates for PDE solutions in which the continuous analysis relies on the fine use of the Sobolev inequality. This problem was already addressed by Antonietti and  Pratelli \cite{MR2754218}  for $p=2$, $N=3$, who proved suboptimal convergence rates.

These kinds of problems have been previously considered in a number of related contexts. In particular, the answer is well known for Poincar\'e's inequality,  related to the first eigenvalue of the Laplacian in a bounded domain $\Omega$ of $\rr^N, N\ge 1$, defined as
\[
\lambda_1(\Omega)=\inf_{u\in H_0^1(\Omega)\setminus\{0\}} \frac{\int_{\Omega}|D u|^2dx}{\int_\Omega u^2dx}.
\] 
 The corresponding finite element approximation is defined as
\[
\lambda_{1h}(V_h)=\inf_{u_h\in V_h\setminus\{0\}} \frac{\int_{\Omega}|D u_h|^2dx}{\int_\Omega u_h^2dx},
\] 
and it is well known to satisfy \cite[Proposition 6.30, p. 315]{MR3405533}, \cite[Section 8, p. 700]{MR1115240}
\[
 \lambda_{1h}(V_h)-\lambda_1(\Omega)\simeq d(\phi , V_h)^2=\inf_{u_h\in V_h} \|\phi-u_h\|_{H^1(\Omega)}^2 \sim h^2,
\]
where $\phi\in H_0^1(\Omega)$ is the first eigenfunction associated with $\lambda_1(\Omega)$. 
In this case, the optimal convergence rate is given by the square of the $H^1$-distance $d(\phi , V_h)$, the distance from the first continuous eigenfunction (normalized to one in $H_0^1(\Omega)$).  This distance turns out to be of order $h$ given that $\phi$ belongs to {$H^2(\Omega)$} for $\Omega$ smooth or convex.

In the context of Sobolev inequalities, we recall the earlier work \cite{MR2739020}, where a different type of inequality is studied, for a bounded convex domain $\Omega\subset\rr^N$ with $N\leq 3$, namely
\[
\gamma=\min_{u\in H^2(\Omega)\cap H_0^1(\Omega),\,u\neq 0}\frac{\|\Delta u\|_{L^2(\Omega)}}{\|u\|_{L^\infty(\Omega)}}.
\]
In \cite{MR2739020}, the authors introduce two numerical methods and present numerical experiments indicating an order-two convergence of the discrete constant $\gamma_h$ toward $\gamma$. Although no analytical proof of this rate is provided there, the computational evidence is consistent and informative. Explicit values of $\gamma$ are available in \cite{MR842406}.

Poincar\'e's inequality can be seen as the simplest instance of the type of problems addressed in this work: determining convergence rates for the finite element approximation of key constants in the functional analysis of PDEs. In the Poincar\'e setting, the minimum is attained by the first eigenfunction of the Laplacian, which is unique up to normalization in $H_0^1(\Omega)$, and the problem is of linear-quadratic nature.

In contrast, the Sobolev inequality presents a genuinely nonlinear character. In bounded domains, the infimum is not attained, although explicit families of minimizing sequences can be constructed via scaling from the exact minimizers in $\mathbb{R}^N$. Our main objective is to develop and combine the necessary tools to fully address this more intricate setting, while building a broader methodology that could also be applied to related problems, such as those involving Hardy constants (see \cite{della2023finite,ignatDCDS2026}).

The main result of this paper is the following optimal convergence rate.
\begin{theorem}\label{main.2}
	Let    $N\geq 2$, $1< p <N$, and  $V_h$ be the space of $P1$ finite elements in the unit ball $B$, associated with a family of regular and quasi-uniform meshes as in Section~\ref{finite.elements}. Then, for $h$ small enough, 	\begin{equation}
\label{est.sob.2}
  S_{h}(p,N)-S(p,N)\simeq  h^{\alpha(p,N)},
\end{equation}
where 
\begin{equation}
\label{alpha}
  \alpha(p,N)=\frac{2(N-p)}{N+p-2}.
\end{equation}
	\end{theorem}

We compare our results with the earlier work of Antonietti and Pratelli \cite{MR2754218}. In the specific case $p = 2$ and $N = 3$, they established the estimate 
\begin{equation}
\label{a.prateli}
  h^{\gamma_0} \lesssim S_h(2,3)-S(2,3)\lesssim h^{\frac 13}, 
\end{equation}
for some exponent $\gamma_0>2\cdot (26)^2/3$. In contrast, we prove the sharp convergence rate $h^{\frac 23}$, which improves upon the upper bound $h^{\frac 13}$ obtained in \cite{MR2754218}. Moreover, our result is shown to be optimal, as we also establish a matching lower bound.

Our proof, which also lays the foundation for a systematic methodology to address similar problems, relies on several novel ingredients.
\begin{itemize}
\item The first key ingredient in our proof is the Sobolev deficit 
\[\delta(u)=\frac{\|D u\|_{L^p(\rr^N)}}{\| u\|_{L^{p*}(\rr^N)}}-S(p,N), \ u\in \dot W^{1,p}(\rr^N)\setminus\{0\}.\]	
Building on the analysis developed in \cite{MR4484209}, we adapt this concept to the $P1$ finite element setting. However, as we will show, this adaptation alone is insufficient, as it yields suboptimal convergence rates.

\item Achieving sharp convergence rates requires the use of quasi-norms introduced in the context of finite element approximations of the $p$-Laplacian \cite{MR1192966, MR2135783}. These quasi-norms play a crucial role in capturing the nonlinear structure of the problem with the necessary accuracy.
\end{itemize}

This general systematic strategy is also applicable to other nonlinear variational constants, such as Hardy constants or the best constants in Gagliardo--Nirenberg inequalities. The first step is to develop a continuous theory that quantifies how the deviation of a minimizing quantity from the optimal constant can be expressed in terms of an appropriate \emph{distance} to the minimizing manifold. The second step is to select a notion of \emph{distance} that faithfully reflects the nonlinear structure of the problem. Finally, one must estimate this distance for specific classes of approximating functions, such as $P1$ finite elements.

Let us recall that in the finite element setting, when approximating solutions $u$ of the continuous $p$-Laplacian problem in a domain $\Omega$, sharp error estimates are usually not derived directly in the classical $W^{1,p}$-norm: standard first-order interpolation estimates in this norm require $W^{2,p}$ regularity of the solution, which is generally not available. Instead, the error is measured in quasi-norms specifically adapted to the function $u$, namely
\[
|u-v|_{p,2}=\left(\int_{\Omega} (|Du|+|D(u-v)|)^{p-2}|D(u-v)|^2dx\right)^{1/2}.
\] 
This tool, introduced in   \cite{MR1192966,MR2135783},  allows one
to handle the degeneracy of the $p$-Laplacian in order to obtain sharp error bounds. 

The key idea underlying the proof of Theorem~\ref{main.2} is to derive optimal convergence rates by exploiting a form of Taylor expansion for the Sobolev deficit:\begin{equation}
	\label{main.idea}
\delta(u)\simeq f({ d}(u,\mathcal{M})), \quad \forall\ u \ \text{with}\ \delta(u)<<1,\, \|Du\|_{L^p(\rr^N)}=1,
\end{equation}
where $f$ is a suitable function and $d$ denotes a suitable distance to the manifold $\mathcal{M}$.

The classical approach (see \cite{MR4484209}) in studying the deficit is to obtain a lower bound of the type
 \begin{equation}
\label{def.low.1}
  \delta(u)\gtrsim f({ d}(u,\mathcal{M})),
\end{equation}
 for  $d(u,v)=\|D(u-v)\|_{L^p(\rr^N)}$ and $f(s)=s^{\max\{2,p\}}$. However, as we shall see below, when proceeding this way, the obtained upper counterpart is of order $\min\{2,p\}$. The derivation of sharp bilateral bounds of the same order requires replacing the Sobolev distance by the quasi-norms above. In fact, in \cite{MR3896203}, for the particular case $p\geq 2$, a variant of the quasi-norm above is also employed. 
 
 One of the main contributions of the present paper is to carefully adapt these methodologies to the distance defined by the quasi-norm: $d(u,v)=|u-v|_{p,2}$ (see Lemma \ref{deficit.below.quasi-norms}).  This, together with the upper bound in Lemma \ref{deficit.above.new}, gives us an estimate of the type \eqref{main.idea}. 

As mentioned above, the results in \cite{MR2754218} are suboptimal due to two key limitations: they rely on a version of \eqref{def.low.1} where the distance is defined by the weaker norm $d(u,v) = |u - v|_{L^{p^*}(\mathbb{R}^N)}$, and they employ only a one-sided estimate for $\delta(u)$, in contrast to the two-sided expansion used in \eqref{main.idea}.

The proof of Theorem~\ref{main.2} is structured as follows.
\begin{itemize}
 \item We begin by establishing the upper bound. To this end, we consider $u_{\lambda,h}\in V_h$, the nodal interpolant of $U_\lambda-U_\lambda|_{\partial B}$, where $U_\lambda \in \mathcal{M}$ (with $\lambda > 0$) is a minimizer. Using the estimate $\delta(u_{\lambda,h}) \lesssim f\big(d(u_{\lambda,h}, U_\lambda)\big)$, we then evaluate the distance $d(u_{\lambda,h}, U_\lambda)$ in terms of the mesh size $h$. Optimizing with respect to the parameter $\lambda$ yields the desired upper bound for the difference $S_h(p,N) - S(p,N)$.
 
 \item To establish the lower bound, we consider the discrete minimizer $u_h \in V_h$ for which $\delta(u_h) = S_h(p,N) - S(p,N)$ is known to be small when $h$ is small. A concentration-compactness argument ensures the existence of a minimizer $U_h \in \mathcal{M}$ such that $\delta(u_h) \gtrsim f\big(d(u_h, U_h)\big)$. The final step consists in deriving a sharp lower estimate for the distance $d(u_h, U_h)$.
 \end{itemize}

When the estimates for the Sobolev deficit are used in the classical framework -- that is, in terms of the $W^{1,p}(\mathbb{R}^N)$-norm, as in \cite{MR4484209} -- rather than via the quasi-norms introduced above, it is still possible to obtain lower and upper bounds. However, these bounds are not sharp. Specifically, this approach yields the following.
	\begin{proposition}\label{main}
	In the setting of Theorem \ref{main.2},
	\begin{equation}
\label{est.sob}
h^{\gamma(p,N)\max\{2,p\}} \lesssim S_{h}(p,N)-S(p,N)\lesssim  h^{\gamma(p,N)\min\{2,p\}},
\end{equation}
where 
\begin{equation}
\label{gamma}
  \gamma(p,N)=\frac{N-p}{N+p(p-2)}.
\end{equation}
\end{proposition}
When $p=2$ and $N\ge 3$, Proposition \ref{main} yields the same sharp result as Theorem \ref{main.2}, namely\[
S_{h}(2,N)-S(2,N)\simeq  h^{\frac{2(N-2)}{N}}.
\]

Our results also lead to some interesting and challenging questions.
\begin{itemize}
\item {\it Sharper asymptotics.} A natural question that arises is whether the lower and upper bounds can be refined to identify the exact limit \[
\lim_{h\rightarrow 0}    h^{-\alpha(p,N)}(S_{h}(p,N)-S(p,N)).
\] 
In a weaker form, one could instead aim to derive sharper, explicit estimates for the multiplicative constants in the asymptotic unilateral bounds, namely

 \[
\lowlim_{h\rightarrow 0}    h^{-\alpha(p,N)}(S_{h}(p,N)-S(p,N))
\] and
\[
\uplim_{h\rightarrow 0} \   h^{-\alpha(p,N)}(S_{h}(p,N)-S(p,N)).
\]
It is worth noting, however, that achieving such refinements may require significant further developments and new technical tools. Moreover, the precise behavior of the asymptotics could depend sensitively on the geometric properties of the finite element mesh.


\item {\it General domains.} Our analysis is limited to the case of the ball $B$ to simplify the presentation. But it can be extended to any bounded, Lipschitz domain which is star shaped with respect to an internal ball. Also the same problem could be formulated and analyzed for more general finite element methods. We refer to \cite[Remarks 3 and 4]{MR2754218} for a review on the possible extensions one could pursue. One could also consider the problem of the finite element approximation in the whole space  $\rr^N$.
\item 
{\it Case $p=1$.} The case $p=1$ is even more interesting since the set of minimizers
is totally different in this case, $\lambda \chi_{B}$, with $\lambda\in \rr{\setminus\{0\}}$ and for some ball $B$. They 
 do not belong to $W^{1,1}(\rr^N)$ but to $BV(\rr^N)$; see \cite{MR2230346,MR2294486}.   
\end{itemize}

The paper is organized as follows.
\begin{itemize}
	\item	In Section~\ref{finite.elements}, we present in detail the finite element method under consideration and establish some preliminary results on the finite element approximation.
	\item	In Section~\ref{sobolev}, we recall known results on the Sobolev deficit and demonstrate how the framework developed in \cite{MR4484209} can be adapted to our setting to prove Theorem~\ref{main.2}. We also derive an upper bound on the Sobolev deficit in terms of quasi-norms (see Lemma~\ref{deficit.above.new}). Furthermore, we analyze the class $\mathcal{M}$ of minimizers in Sobolev's inequality (see \cite{MR0448404, MR0463908}) to obtain estimates of their Sobolev norms inside and outside the unit ball and to quantify their distance to the finite element subspace $V_h$.
	\item	Section~\ref{proofs} contains the complete proofs of the main results. We begin with the proof of Proposition~\ref{main}, which addresses the core difficulties of the problem, and proceed to establish the sharp convergence result stated in Theorem~\ref{main.2}.
	\item	Finally, Appendix~\ref{app} collects the proofs of several technical results that play a crucial role in the main arguments of the paper.
	\end{itemize}
	
\section{Preliminaries on finite elements}
\label{finite.elements}

 \subsection{The finite element basis}
Let us consider  $\Omega$  a polyhedral domain in $ \rr^N$. Of course, depending on the dimension, we should refer to intervals in dimension $N=1$, polygons in dimension $N=2$, polyhedra in dimension $N=3$ and polytopes in arbitrary dimension $N$. But to simplify the presentation, we will generically use the term ``polyhedron" without distinguishing the dimension. For each positive $h$, we construct a partition $\th$   of the domain $\Omega$ into a {finite set of $N$-simplicial elements } $T$ satisfying the following properties (\cite[pp.~38 and 51]{ciarlet}; see also \cite[Appendices B and C]{volker} for a concise presentation):
\begin{enumerate}
\item $\cup _{T\in \th} T=\overline{\Omega}$.
\item Each  $T\in \mathcal{T}_h$ is closed and its interior  nonempty.
\item For distinct $T$ and $T'$, their interiors are disjoint.
\item If $T,T'\in \th$, $T\neq T'$, then either $T\cap T^{'}=\emptyset $ or $T\cap T^{'}$ is a common $m$-face, $m\in \{0,\dots,N-1\}$.
 \end{enumerate}

For each $T\in \th$, we denote by $\rho_T$ and $h_T$  the diameter of the largest ball contained in $T$ and the diameter of $T$, respectively. We set 
\[
h=h(\th)=\max_{T\in \th}h_T.
\] 
We will consider a set of regular meshes $(\th)_{h>0}$:  there exists $\sigma>0$, independent of $h$, such that
\begin{equation}
\label{regular}
  \frac{h_T}{\rho_T}\leq\sigma, \ \forall \ T\in \th, \ \forall\ h>0.
\end{equation}
Also the mesh is assumed to be quasi-uniform, that is, 
\begin{equation}\label{quasi-uniform}
  \inf_{h>0}\frac {\min_{T\in \th}h_T}{ \max _{T\in \th} h_T}>0.
\end{equation}

Each element of the mesh $\th$ is the image of a reference simplex (an interval $(N=1)$, a triangle $(N=2)$, a tetrahedron $(N=3)$, or an $N$-simplex, in general) through an affine mapping $F_T:\rr^N\rightarrow \rr^N$, 
\[
F_T(\widehat x)=B_T \widehat x+b_T,
\]
$B_T$ being an invertible $N\times N$ matrix, $b_T\in \rr^N$, that is, 
\[
F_T(\widehat T)=T, \quad \ \forall \ T\in \th.
\]
Matrix $B_T=\nabla F_T$ satisfies  (cf. \cite[Lemma C.~12, p.~735]{volker})
\begin{equation}\label{bt}
  \|B_T\|\leq \frac{h_T}{\rho_{\widehat T}}, \quad \|B_T^{-1}\|\leq \frac{h_{\widehat T}}{\rho_{T}}
\end{equation}
and 
\[
|\det {B_T}|=|  J F_T|=\frac{|T|}{|\widehat{T}|}.
\]

For a fixed partition or grid $\th$, we define the finite element space $V_h$ as 
\[
 V_h=\{ f\in C(\overline\Omega); f\circ F_T\in \mathbb{P}^1(\widehat T), \ \forall T\in \th, \, f =0 \,  \hbox{ on } \, \partial \Omega\},
\]
where $ \mathbb{P}^1(\widehat T)$ is the space of linear  polynomials  on $\widehat T$.

Our analysis is carried out in   the unit ball, $
\Omega=B$. This requires a first approximation of $B$ by means of convex polyhedral domains $B_h\subset B$ (as in \cite{MR2754218}) such that all the nodes of $\partial B_h$ are on $\partial B$. The finite element subspace $V_h$ employed to define the numerical approximation of the Sobolev constant will be the one corresponding to the polyhedral domain $B_h$ so that $V_h\subset H^1_0(B_h)\cap C(\overline B_h)$. The elements of $V_h$ can be extended by zero outside $B_h$ so that $V_h$  can also be viewed as a finite-dimensional subspace of $H^1_0(B)$. Since the boundary faces of $B_h$ have their vertices on $\partial B$ and, by \eqref{regular}, circumradius at most $Ch$, the convexity of $B_h$ implies that $\{|x|<1-h\}\subset B_h$ for $h$ small enough.


\subsection{Approximation by finite elements}
We first recall some classical finite element approximation results in Sobolev spaces \cite[Theorem 4.4.20]{MR2373954}: for any polyhedral domain $\Omega$ in $ \rr^N$,   and partition $\mathcal{T}_h$ as above (see  \cite[Theorem 4.4.4]{MR2373954} for the complete set of restrictions), the global piecewise linear interpolant $I^h$, mapping $W^{2,p}(\Omega)$ into the space of continuous piecewise linear functions on $\mathcal{T}_h$ (with $I^hu\in V_h$ if $u=0$ on $\partial\Omega$), satisfies   \begin{equation}
\label{fem.error.sobolev}
  \Big(\sum_{T\in \mathcal{T}_h }\| u-I^h u\|^p_{W^{s,p}(T)} \Big)^{\frac 1p}\leq Ch^{2-s}\|u\|_{W^{2,p}(\Omega)}, \ \forall u\in W^{2,p}(\Omega),
\end{equation}
for all $s\in \{0,1\}$ and $p>N/2$.

Of course, the multiplicative constant $C$ in \eqref{fem.error.sobolev} depends on $N$, $p$, $s$ and the domain $\Omega$, but it is independent of $u$. All along the paper,  the dependence of the constant $C$ on each of the parameters of the problem (in particular the dimension $N$ and the exponent $p$) will not be made explicit in the notation. Constants in our estimates are normally independent of the function $u$ under consideration (or its discrete counterparts), unless otherwise stated.

The above restriction on $p >N/2$ is necessary when dealing with general functions in  $W^{2,p}(\Omega)$.  The fact that $p >N/2$ ensures the continuous embedding of $W^{2,p}(\Omega)$ into $C^0(\bar \Omega)$ needed to define the interpolation $I^h$ by taking the pointwise values over the vertices of the mesh. 
{The above result can be adapted to the particular case of $W^{2,\infty}(\Omega)$ functions. It can be established by a slight adaptation of the arguments in  \cite[Chapter 4, Theorem~4.4.4]{MR2373954}. The result can in fact be extended to more general finite element spaces and to functions  $u\in C^m(\overline{\Omega})$, but such extensions fall outside the scope of the present paper.
}
\begin{lemma}\label{est.c2}Let $\mathcal{T}_h$ be a regular mesh on a 
	polyhedral domain $\Omega\mathrel{\subset} \rr^N$, $N\ge 1$, and    $s\in \{0,1\}$. Then, for all $1\le p < \infty$ and $|\alpha|=s$, there  exists a positive constant $C=C(N,p,s,\sigma)$ such that
	\begin{equation}
\label{est.fem.c2}
    \Big(\sum_{T\in \mathcal{T}_h }\| D^\alpha(u-I^h u)\|^p_{L^{p}(T)} \Big)^{\frac 1p}\leq Ch^{2-s}\Big( \sum_{T\in \mathcal{T}_h} |T| \|D^2 u\|^p_{L^\infty(T)}\Big)^{\frac 1p}, \ \forall u\in W^{2,\infty}(\Omega), 
\end{equation}
and
	\begin{equation}
\label{est.fem.c2.infty}
    \max_{T\in \mathcal{T}_h}\| D^\alpha(u-I^h u)\|_{L^{\infty}(T)}\leq Ch^{2-s}\max_{T\in \mathcal{T}_h}   \|D^2 u\|_{L^\infty(T)}, \ \forall u\in W^{2,\infty}(\Omega).
\end{equation}
Moreover, for any $T\in \mathcal{T}_h$ and $p\geq 1$,
\begin{equation}
\label{est.fem.c3.loc}
  \|D^\alpha I^hu\|_{L^p(T)}\lesssim h_T^{-|\alpha|+\frac Np}\|u\|_{L^\infty(T)}, \ \forall \ u\in C^1(T).
%
\end{equation}

\end{lemma}


 The above upper bounds are optimal when considering functions that are strictly convex in one direction. 
 
We first recall the following result, a consequence of \cite[Lemma~2.5]{MR2261017}, which will be instrumental: for any    $1\leq p<\infty$, any  strictly convex function $u\in C^2(\Omega)$ and any simplex $T\in \mathcal{T}_h$ with vertices $\{a_k\}_{k=1}^{N+1}$, the interpolation error of $u$ by its linear interpolant $I^h u$ admits the following lower bound: \[
 \|u-I^h u\|^p_{L^p(T)}\geq C_p  \min_{x\in  T} |\lambda_1(D^2u)(x)|^p |T| \sum_{1\leq i\neq j\leq N+1}|a_i-a_j|^{2p}.
 \]  
 Here $\lambda_1(D^2u)(x)$ denotes the eigenvalue of $D^2u(x)$ with the smallest absolute value. In particular, when $\mathcal{T}_h$ is quasi-{uniform}, we have
 \[
 \|u-I^h u\|^p_{L^p(T)}\geq C_p h^{N+2p}  \min_{x\in   T} |\lambda_1(D^2u)(x)|^p.
 \]

We establish a similar estimate for the gradients. {Denoting the unit sphere in $\mathbb{R}^N$ by $\mathbf{S}^{N-1}$, we obtain the following result.}
\begin{lemma}\label{lower.bound.triangle}For any $p\in (1,\infty)$, there exists a positive constant $C(p)$ such that for any $T\in \mathcal{T}_h$ and any $u\in C^2(T)$,  
	\[
	 \min_{A\in \rr^N} \int_T |Du-A|^pdx \geq C(p)\rho_T^{N+p} \max _{\xi\in \mathbf{S}^{N-1}}\min _{x\in T}|\xi^T D^2u(x)\xi|^p.
	\]
	\end{lemma}
	The proof will be given in Appendix~\ref{app}. 
	
\subsection{Eigenvalue approximation and the case $p=2$}
\label{classical.approximation} 

Let us first recall the classical estimates for the eigenvalues of the Laplacian, \cite{MR1115240}. We focus on the simplest case since we aim to present those tools that will be useful to handle the Sobolev constant for $p=2$.  We refer to \cite{MR1115240} for other extensions.  

We adopt the notation in \cite[Section 8, p.~697]{MR1115240}. Let $V$ be a real Hilbert space and $a(\cdot, \cdot)$ be a symmetric, continuous and coercive bilinear form on $V$. Let $H$ be another Hilbert space such that $V\subset H$, with  compact embedding, $b$ a continuous symmetric bilinear form on $H$, satisfying $b(u,u)>0$ for all $u\in V$, $u\neq 0$. Let $V_h\subset V$ be a family of finite-dimensional subspaces of $V$. 

Let $\lambda_1$ be the first eigenvalue of \textit{the form $a$, relative to the form $b$}. Then,  there exists a nonzero $\phi_1\in  V$ such that
\[
a(\phi_1,v)=\lambda_1\,{b(\phi_1,v)}, \forall \ v\in V.
\]
In a similar way, we define $\lambda_{1h}$ in the space $V_h$:
\[
a(\phi_{1h}, {v_h})=\lambda_{1h}\,{b(\phi_{1h},v_h)}, \forall \ v_h\in V_h.
\]
An important result in the eigenvalue  approximation theory  (established by Babu\v{s}ka--Osborn  \cite[Section 8, p.~700]{MR1115240}, which goes back to Chatelin \cite[Proposition~6.30, p.~315]{MR3405533}) asserts that
\begin{equation}
\label{two-side.estimate}
  C_1\eps_h^2 \leq \lambda_{1h}-\lambda_1\leq C_2\eps_h^2
\end{equation}
where
\[
\begin{aligned}
\eps_h&=\inf\big\{d_V(\phi,V_h):\ \phi\in V,\ \|\phi\|_V=1,\ a(\phi,v)=\lambda_1 b(\phi,v)\ \forall\, v\in V\big\},\\
d_V(\phi,V_h)&=\inf_{u_h\in V_h } \|\phi-u_h\|_{V},
\end{aligned}
\]
that is, $\eps_h$ is the distance from $V_h$ to the normalized eigenfunction associated with $\lambda_1$ that is best approximated in $V_h$. When $\lambda_1$ is simple, $\eps_h=d_V(\phi_1,V_h)$ with $\|\phi_1\|_V=1$, and \eqref{two-side.estimate} is stated in \cite[p.~700]{MR1115240}. When $\lambda_1$ is multiple, the lower bound in \eqref{two-side.estimate} is given in \cite[(8.48)]{MR1115240} (see also \cite[Proposition~6.30]{MR3405533}) and the upper bound in \cite[Theorem~9.1]{MR1115240}.
The upper bound, together with the trivial estimate $0\leq \lambda_{1h}-\lambda_1$,  can also be found in \cite[Theorem 6.4-2, p.~148]{raviart}.
If the family $(V_h)_{h>0}$ satisfies  
\[
\forall\,  u\in V, \ \lim_{h\rightarrow 0}\inf_{v_h\in V_h } \|u-v_h\|_{V}=0,
\]
one obtains that $\eps_h\rightarrow 0$.  But the derivation of the convergence order of $\lambda_{1h}$ toward $\lambda_1$ requires a finer analysis.

\section{The Sobolev constant and finite element approximation of minimizers}\label{sobolev}

This section is devoted to gathering a number of technical results, most of whose proofs are postponed to Appendix~\ref{app} at the end of the paper.

Let us recall some classical facts about the Sobolev inequality: for any $1<p<N$, $N\geq 2$, it holds
\begin{equation}
\label{sobolev.ineq}
  \| D u\|_{L^p(\rr^N)}\geq S(p,N)\|u\|_{L^{p^*}(\rr^N)}, \, \forall \, u\in \dot W^{1,p}(\rr^N),
\end{equation}
  where $p^*=pN/(N-p)$. The optimal constant $S(p,N)$ was obtained by Aubin \cite{MR0448404} and Talenti  \cite{MR0463908} and shown to be attained for the $(N+2)$-dimensional manifold $\mathcal{M}$ of functions of the form
\begin{equation}\label{Talentian1}
U_{c,\lambda,x_0}(x)=c U_{\lambda,x_0}(x)=c\lambda U(\lambda^{\frac p{N-p}}(x-x_0)),  \end{equation}   with
    \begin{equation}\label{Talentian2}
    U(x)=u_0(|x|); \, u_0(r)=\frac{k_0}{(1+|r|^{\frac{p}{p-1}})^{\frac{N-p}p}}.
    \end{equation}
where  $k_0$ is chosen such that $\| D U\|_{L^p(\rr^N)}=1$. In particular, when $c=1$, we have
$\|DU_{\lambda,x_0}\|_{L^p(\rr^N)}=1$. 

The $N+2$ free parameters of the manifold $\mathcal{M}$ are the center $x_0 \in \rr^N$, the positive parameter $\lambda >0$ and the constant $c\in \rr{\setminus\{0\}}$.

\subsection{The $p$-Sobolev deficit} \label{deficit} Let us  now consider the so-called $p$-Sobolev deficit
\[
\delta(u)=\frac{  \| D u\|_{L^p(\rr^N)}}{\|u\|_{L^{p*}(\rr^N)} }-S({p,N}), \ \forall \ u\in \dot W^{1,p}(\rr^ N)\setminus\{0\}.
\]
It is well known from the works of Bianchi and Egnell \cite{MR1124290} and more recently of Figalli and Zhang \cite{MR4484209} that for every $1<p<N$, there exists a positive constant   $c(p,N)$ such that
\begin{equation}
\label{deficit.below}
  \delta(u)\geq c(p,N) \inf_{v\in \mathcal{M}} \Big(\frac{\|Du-Dv\|_{L^p(\rr^N)}}{\|Du\|_{L^p(\rr^N)}}
\Big)^{\max\{2,p\}}, \ \forall \ u \in {\dot W^{1,p}(\rr^N)} .
\end{equation}

Indeed, the results in \cite{MR4484209} include a refined estimate in the range  $2\le p<N$, which is particularly useful for our analysis. This estimate involves a quasi-norm rather than the classical Sobolev seminorm $\|Du-Dv\|_{L^p(\rr^N)}$ and is reminiscent of the approach introduced in \cite{MR1192966, MR2135783} to obtain sharper convergence rates for the $p$-Laplacian. The following lemma presents a refined version of these estimates, valid for the full range $1<p<N$, expressed in terms of the quasi-norms used in finite element theory.

\begin{lemma}\label{deficit.below.quasi-norms} Let $1<p<N$. 
	There is a positive constant $c(p,N)$   such that for any $\ u\in \dot W^{1,p}(\rr^N)$, there exists a function $v=v_u\in \mathcal{M}$ such that
	\begin{equation}\label{deficit.below.2}
		\frac{\delta(u)}{c(p,N)}\geq \frac{\int_{\rr^N}(|Du-Dv|+|Dv|)^{p-2}|Du-Dv|^{2}\,dx}{\|Du\|^p_{L^p(\rr^N)}} + \Big(\frac{\|Du-Dv\|_{L^p(\rr^N)}}{\|Du\|_{L^p(\rr^N)}}
\Big)^{\max\{2,p\}}.
\end{equation}
\end{lemma}

\begin{remark}
	In particular, we obtain 
	\begin{equation}
 \delta(u)\geq c \inf_{v\in \mathcal{M}}   \frac{\int_{\rr^N}(|Du-Dv|+|Dv|)^{p-2}|Du-Dv|^{2}\,dx}{\|Du\|^p_{L^p(\rr^N)}}, \ \forall \ u\in \dot W^{1,p}(\rr^N).
	\end{equation}
\end{remark}

 After submitting this paper, we were informed that a related result for $2 < p < N$ appears in \cite[Theorem~14]{frank2024}.

\begin{proof}
	The proof follows the lines of Theorem 1.1 in  \cite{MR4484209}. 
	
	By homogeneity, we can assume that $\|Du\|_{L^p(\rr^N)}=1$. Also it is sufficient to consider the case where $\delta(u)<<1$. (If $\delta(u)\geq\delta_0>0$, then \eqref{deficit.below.2} holds with $v=U$, since its right-hand side is then bounded by a constant $C(p)$, provided $c(p,N)\leq\delta_0/C(p)$.) It is shown in \cite[Lemma 4.1]{MR4484209} that when the deficit is small, there exists  
	   $v\in \mathcal{M}$ such that $\eps:=\|Du-Dv\|_{L^p(\rr^N)} $ is small  
	 and $u-v$ is orthogonal to $T_v\mathcal{M}$ in $L^2(\rr^N;|v|^{p^*-2}dx)$, i.e. $\int_{\rr^N}|v|^{p^*-2}\xi\,(u-v)\,dx=0$ for all $\xi\in T_v\mathcal{M}$ (see \cite[Section~3]{MR4484209}). For full details, we refer the reader to the proofs of Theorem 1.1 and Lemma 4.1 in \cite{MR4484209}.  
	
	For $1<p\leq 2N/(N+2)$,  \cite[formulas (4.4) and (4.5)]{MR4484209} give 
\begin{align}\label{figalli.p<2}
  	\delta(u)&\geq c_0(p,N) \int _{\rr^N} \min\{ |D(u-v)|^p, |Dv|^{p-2}|D(u-v)|^{2}\}dx\\
  	&= c_0(p,N) \int _{\rr^N} \min\{ |D(u-v)|^{p-2}, |Dv|^{p-2}\}|D(u-v)|^{2}dx.\nonumber\\
  	& \geq c_1(p,N) \|D(u-v)\|_{L^{p}(\rr^N)}^2.\nonumber
\end{align}
When $p<2$, we also get
\begin{align}
\label{quasi-norm.ineg}
\delta (u)&\geq 	c_0(p,N) \int _{\rr^N} \min\{ |D(u-v)|^{p-2}, |Dv|^{p-2}\}|D(u-v)|^{2}dx\\
&\geq c_0(p,N) \int _{\rr^N}  (|D(u-v)|+ |Dv|)^{p-2}|D(u-v)|^{2}dx.\nonumber
\end{align}
Both inequalities yield  \eqref{deficit.below.2} when $1<p\leq 2N/(N+2)$.

When $p\in (2N/(N+2), 2)$, repeating the arguments in \cite{MR4484209} (see the case $2N/(N+2)<p<2$ in the proof of Theorem~1.1 there),
we have
\begin{align*}
	\delta(u)&\geq c_0(p,N) \int _{\rr^N} \min\{ |D(u-v)|^{p-2}, |Dv|^{p-2}\}|D(u-v)|^{2}dx-
C_1 \int _{\rr^N}|u-v|^{p^*}dx\\
&\geq \frac{c_0(p,N)}{2} \int _{\rr^N} \min\{ |D(u-v)|^{p-2}, |Dv|^{p-2}\}|D(u-v)|^{2}dx+\frac{c_1(p,N)}{4}\|D(u-v)\|_{L^p(\rr^N)}^{2}\\
&\quad +\frac{c_1(p,N)}{4}\|D(u-v)\|_{L^p(\rr^N)}^{2} -C_1 \|u-v\|_{L^{p^*}(\rr^N)}^{p^*}. 
\end{align*}
Since $p^*>2$, by the Sobolev inequality the sum of the last two terms on the right-hand side is nonnegative if $\|D(u-v)\|_{L^p(\rr^N)}$ is small.  
 
For $p\geq 2$, the proof of Theorem 1.1 of \cite{MR4484209}, keeping the term $\int_{\rr^N}|Dv|^{p-2}|D(u-v)|^2dx$ with a positive coefficient (this is possible since the constant $\lambda$ in \cite[Proposition~3.8]{MR4484209} is positive), gives 
\begin{align*}
\delta(u)&\geq c_0(p,N) \Big(\int _{\rr^N} |Dv|^{p-2}|D(u-v)|^2\,dx+ \int _{\rr^N}  |D(u-v)|^p\,dx\Big)\\
&\geq c_1(p,N) \int _{\rr^N} (|D(u-v)|+|Dv|)^{p-2}|D(u-v)|^2\,dx.
\end{align*}
This completes the proof.

\end{proof}

We now turn our attention to the upper bounds of the deficit.
 
 \begin{lemma}
	\label{deficit.above.new}
	Let $1<p<N$.  There exists a positive constant $C(p,N)$ and a positive number $\eps_0$ such that
	\begin{equation}
		\label{eq.deficit.above}
		\frac{\delta(u)}{C(p,N)}\leq 
		 	 \frac{	\int_{\rr^N}(|Du|+|D(u-v)|)^{p-2}|Du-Dv|^{2}\,dx}{\|Du\|^p_{L^p(\rr^N)}}, \quad  p\geq 2,
		 	 	\end{equation}
	and
		\begin{equation}
		\label{eq.deficit.above.2}
			\begin{split}
			\frac{\delta(u)}{C(p,N)}\leq
					 	 &\frac{	\int_{\rr^N}(|Du|+|D(u-v)|)^{p-2}|Du-Dv|^{2}\,dx}{\|Du\|^p_{L^p(\rr^N)}} \\
					 	 &+\Big(\frac{ \int_{\rr^N}|Dv|^{p-1}  |D(u-v)|\,dx}{\|Du\|^p_{L^p(\rr^N)}}\Big)^2,\quad  1<p<2,
			\end{split}
			\end{equation}
hold for all  $u\in {\dot W^{1,p}(\rr^N)}$ with $\delta(u)<S(p,N)$ and $v\in \mathcal{M}$ such that $\|D(u-v)\|_{L^p(\rr^N)}\leq \eps_0 \|Du\|_{L^p(\rr^N)}$.
\end{lemma}

\begin{remark}\label{estimate.general.p}
			We shall prove below that, for any $1<p<N$, the following inequality holds:
\begin{equation}
\label{general.not.well}
  	\delta(u)\leq \frac{C(p,N)}{\|Du\|^p_{L^p(\rr^N)}}	\Big(\int_{\rr^N}(|Dv|+ |D(u-v)|)^{p-2}|D(u-v)|^2 dx +\int_{\rr^N}|Dv|^{p-2}|Du-Dv|^{2}dx\Big).
\end{equation}
This inequality immediately leads to estimate  \eqref{eq.deficit.above}.

However, when $1<p<2$, as we shall see below (see Remark \ref{singularremark}), the last term on the right-hand side, namely
$$\int_{\rr^N}|Dv|^{p-2}|Du-Dv|^{2}dx$$
is singular. Thus, it has to be replaced by
\[
\frac{1}{\|Du\|^p_{L^p(\rr^N)}}\left( \int_{\mathbb{R}^N} |Dv|^{p-1} |D(u - v)| \, dx \right)^2
\]
(with the normalization of \eqref{eq.deficit.above.2}), which is also well defined in the range $1<p<2$ when $u\in V_h$ and $v\in \mathcal{M}$.

\end{remark}

As a consequence of this lemma, we can deduce an upper counterpart to the lower bound in \eqref{deficit.below}. Although one might expect that an upper bound for the deficit would be easier to obtain than the lower bound, we have not found such a result in the literature. This estimate will play an important role in the proof of Proposition \ref{main}.

\begin{cor}
	\label{deficit.above} 
	There are a constant $C=C(p,N)$  and a positive number $\eps_0$ such that
	\begin{equation}\label{upperboundx}
 \delta(u)\leq C \inf_{v\in \mathcal{M}}  \Big(\frac{\|Du-Dv\|_{L^p(\rr^N)}}{\|Du\|_{L^p(\rr^N)}}
\Big)^{\min\{2,p\}}.
\end{equation} holds for all $u\in {\dot W^{1,p}(\rr^N)}$  satisfying
\[
d(u,\mathcal{M})=\inf_{v\in \mathcal{M}} \|Du-Dv\|_{L^p(\rr^N)}\mathrel{<} \eps_0 \|Du\|_{L^p(\rr^N)}.
\]
\end{cor}
%

\begin{remark}
	This estimate is sharp. In fact, the same argument as in \cite[Remark~1.2]{MR4484209} shows that if \begin{equation}
\label{deficit.above.2} \delta(u)\leq C \inf_{v\in \mathcal{M}}  \Big(\frac{\|Du-Dv\|_{L^p(\rr^N)}}{\|Du\|_{L^p(\rr^N)}}
\Big)^{\beta}
\end{equation} holds, 
then, necessarily, $\beta\leq \min\{2,p\}$.

Note that, for $p\neq 2$, there is a gap in the exponents:  $\max\{2,p\}$ for the lower bound in \eqref{deficit.below}, while  the upper bound in \eqref{upperboundx} holds with $\min\{2,p\}$.
\end{remark}

%

\subsection{\texorpdfstring{Estimates on the minimizers in $\mathcal{M}$}{Estimates on the minimizers in M}}
In this section, we establish various estimates for the minimizers $U_{\lambda,x_0}\in \mathcal{M}$ that, as indicated above, are defined as in \eqref{Talentian1} and \eqref{Talentian2}. Their proofs are given in Appendix~\ref{app}.

\begin{lemma}
	\label{unif.est.U.lambda} Let $1<p<N$ and $N\geq 2 $. There exist two positive constants $c_1=c_1(N,p)$ and $c_2=c_2(N,p)$ such that for any $|x_0|<1$, the minimizer $U_{\lambda,x_0}$ 	
	 satisfies	
\begin{equation}
\label{est.lambda.mic}
  c_1 \lambda^{\frac{p}{N-p}(N+\frac p{p-1})}\leq \int_{|x|<1} |DU_{\lambda,x_0}|^pdx\leq c_2 \lambda^{\frac{p}{N-p}(N+\frac p{p-1})}, \ \forall \ 0<\lambda< 1,
\end{equation}
\begin{equation}
\label{est.lambda.mare}
  1- \frac{c_2\lambda^{-\frac{p}{p-1}}}{ (1-|x_0|)^{\frac{N-p}{p-1}}}\leq  \int_{|x|<1} |DU_{\lambda,x_0}|^pdx\leq 1-  {c_1}\lambda^{-\frac{p}{p-1}},\ \forall \ \lambda>1.
\end{equation}
	\end{lemma}
	
	 \begin{remark}\label{x0mare}
	Observe that, when $|x_0|\geq 1$,  the exterior of the unit ball   contains  a half-space passing through $x_0$. Thus, 
\[
 \int_{|x|\geq 1}|DU_{\lambda,x_0}|^pdx\geq \frac 12 \int_{\rr^N}|DU_{\lambda,x_0}|^pdx=\frac 12
\]
for all   $\lambda>0$. This illustrates behavior opposite to the energy concentration phenomenon described by \eqref{est.lambda.mare} in the case $|x_0|<1$.
\end{remark}

\begin{lemma}\label{est.second.derivative}
Let $1<p<N$.
The second-order derivatives of the minimizer $U_{\lambda,x_0}$ satisfy the following estimates:
\begin{equation}
\label{est.superior.D2U}
|D^2U_{\lambda,x_0}(x)|\leq C_{N,p} \lambda^{\frac{N+p}{N-p}} a(\lambda^{\frac{p}{N-p}}|x-x_0|), \ x\in \rr^N\setminus\{x_0\},
\end{equation}
where
\[
a(r)=r^{\frac {2-p}{p-1}} \Big(  1+r^{\frac p{p-1}}\Big)^{-\frac Np }, \quad {\text{so that}\quad \frac{|u_0'(r)|}{r}=|c_{N,p}|\,a(r)\ \text{ ($c_{N,p}$ as in Appendix~\ref{app.D2U})}}
\]
and
\begin{equation}
\label{est.bila.unitate}
  \int_{|x|<1} |D^2 U_{\lambda,x_0}(x)|^pdx\leq C_{N,p} \lambda^{\frac {p^2}{N-p}}.
\end{equation}

{Moreover, there exist constants $\eta=\eta(N,p)\in(0,1)$ and $A_{N,p}>0$, such that, for all $\lambda>0$, $x_0\in\rr^N$, $x\neq x_0$ and $\xi\in\mathbf{S}^{N-1}$ with $|\xi\cdot(x-x_0)|\le\eta|x-x_0|$,
\begin{equation}
\label{est.inferior.D2U}
|\xi^TD^2U_{\lambda,x_0}(x)\xi|\geq {A_{N,p}} \lambda^{\frac{N+p}{N-p}}       a(\lambda^{\frac p{N-p}}|x-x_0|).
\end{equation}}
\end{lemma}

The proofs of these results are presented in Appendix~\ref{app}.

\subsection{Approximation of the minimizers}\label{approx.proper.subspaces} 

As discussed in Section \ref{classical.approximation}, it is necessary to estimate the distance between the minimizers in  $\mathcal{M}$  and the space  $V_h$,  the $P1$ finite element  space introduced in Section \ref{finite.elements}.

To obtain the upper bounds in Proposition \ref{main} and Theorem \ref{main.2}, we do not need to measure how far all the minimizers in \( \mathcal{M} \) are from the space \( V_h \); it suffices to sample just one of them. In the next lemma, we estimate the distance between the minimizer \( U_{\lambda,0} \) and its finite element approximation in \( V_h \) defined below, both in the classical Sobolev norm and in the quasi-norms. Since the upper bound for the Sobolev deficit obtained in Lemma \ref{deficit.above.new} includes an additional term when \( 1 < p < 2 \), we also need to estimate this remainder term.

\begin{lemma}\label{est.distance.vh}
Let   $N\geq 2 $ and $1<p<N$. 
	For any $\lambda>1$ and $h<1/2$ such that $h\lambda^{\frac p{N-p}}\leq 1$, the following holds:
	\begin{equation}
\label{est.u.liniar.1}
    \int_{\rr^N} |D(U_{\lambda,0} -u_h)|^pdx \lesssim \lambda^{-\frac {p}{p-1}}+(h\lambda^{\frac{p}{N-p}})^p,
\end{equation}
	\begin{equation}
\label{est.u.liniar.2}
    \int_{\rr^N} (|DU_{\lambda,0} |+|D(U_{\lambda,0} -u_h)|)^{p-2} |D(U_{\lambda,0} -u_h)|^2dx \lesssim \lambda^{-\frac {p}{p-1}}+(h\lambda^{\frac{p}{N-p}})^2, 
\end{equation}
	\begin{equation}
\label{est.u.liniar.3}
    \int_{\rr^N} |DU_{\lambda,0} |^{p-1} |D(U_{\lambda,0} -u_h)|dx \lesssim \lambda^{-\frac {p}{p-1}}+ h\lambda^{\frac{p}{N-p}} , 1<p<2, 
\end{equation}
where $u_h\in V_h$ is given by $u_h=I^h(U_{\lambda,0}-(U_{\lambda,0})_{|\partial B})$ extended by zero outside $B_h$.
	\end{lemma}
\begin{remark}\label{upper.distance.U.Vh}
	The above lemma implies 
\[
  \inf_{u_h\in  V_h} \|D(U_{\lambda,0}-u_h)\|_{L^p(\rr^N)}\lesssim \lambda^{-\frac {1}{p-1}}+h\lambda^{\frac{p}{N-p}}.
\]
Optimizing this estimate  with  $\lambda^{-\frac 1{p-1}}=h^{\frac{N-p}{N+p^2-2p}}$, we get 
\[
    \inf_{u_h\in  V_h} \|D(U_{\lambda,0}-u_h)\|_{L^p(\rr^N)} \lesssim h^{\gamma(p,N)}
\]
where
\[\gamma(p,N)=\frac{N-p}{N+p(p-2)}.\]

A similar argument, with $\lambda$ now chosen such that $\lambda^{-\frac p{p-1}}=(h\lambda^{\frac p{N-p}})^2$, that is, $\lambda=h^{-\frac{2(p-1)(N-p)}{p(N+p-2)}}$, works for the quasi-norms 
\[
  \inf_{u_h\in  V_h}\int_{\rr^N} (|DU_{\lambda,0} |+|D(U_{\lambda,0} -u_h)|)^{p-2} |D(U_{\lambda,0} -u_h)|^2dx \lesssim h^{\alpha(p,N)}
\]
with
\[
\alpha(p,N)=\frac{2(N-p)}{N+p-2}.
\]
\end{remark}

 \section{Proof of the main results}\label{proofs}
 In this section, we first prove  Proposition \ref{main}:   
\begin{equation}\label{maininequality1}
h^{\gamma(p,N)\max\{2,p\}} \lesssim S_{h}(p,N)-S(p,N)\lesssim  h^{\gamma(p,N)\min\{2,p\}},
\end{equation}
and secondly, we consider  the improved result in Theorem \ref{main.2}:
\begin{equation}\label{maininequality2}
h^{\alpha(p,N) } \lesssim S_{h}(p,N)-S(p,N)\lesssim  h^{\alpha(p,N) }.
\end{equation}
The proof of Theorem~\ref{main.2} follows the same overall strategy as that of Proposition~\ref{main}, with the key difference that the Sobolev distances used in Proposition~\ref{main} are replaced by quasi-norms of the finite element analysis of the $p$-Laplacian. This allows us to obtain sharp lower and upper bounds of the same asymptotic order in $h$.

The proof of Proposition~\ref{main} (and therefore that of Theorem~\ref{main.2} as well) is conducted in two main steps:
\begin{itemize}
\item Step 1. In the first step, we prove the upper bound in \eqref{maininequality1} (or in \eqref{maininequality2}). For that, we employ the finite element interpolant of a suitable continuous minimizer $U_\lambda$ (corrected by a constant so that it vanishes on $\partial B$) by optimally choosing the value of $\lambda \mathrel{\gg} 1$ as a function of the mesh size $h\mathrel{\ll} 1$.

\item Step 2. In the second step, we prove the lower bound. Employing the general lower bound for the deficit in \eqref{deficit.below.2},  this is done by a suitable approximation in ${\mathcal M}$ of the finite element minimizer.
\end{itemize}

We consider the unit ball $B$.  The Sobolev constant in the unit ball $B$ coincides with that in the whole space, although it is not attained in $B$; see \cite{MR2754218}.

We construct the space $V_h$ as in Section \ref{finite.elements} and  use that
\[
S_{h}(p,N)=\min_{u_h\in V_h\setminus\{0\}}\frac{\|Du_h\|_{L^p(B)}}{\|u_h\|_{L^{p^*}(B)}}
 =\min_{u_h\in V_h\setminus\{0\}}\frac{\|Du_h\|_{L^p(\rr^N)}}{\|u_h\|_{L^{p^*}(\rr^N)}},
\]
identifying $V_h$ with the space of functions defined everywhere by extending them by zero outside their support.

	For notational simplicity, throughout this section, we will write $S_h$ and $S$ in place of the full expressions for the two constants introduced above.
  

\begin{proof}[Proof of Proposition \ref{main}] 
As described above, we proceed in two main steps.

\noindent {\textit{Step} I. \textit{The upper bound.} } 
%
%
%
To obtain an upper bound, we choose a suitable $U_{\lambda}$ and the finite element approximation $u_{h,\lambda}=I^h(U_{\lambda}-{U_{\lambda}}_{|_{\partial B}})\in V_h$ extended by zero outside $B_h$, with a suitable $\lambda >>1$ and $h<<1$.    

Lemma \ref{est.distance.vh}  shows that choosing $$\lambda^{\frac 1{p-1}}=h^{-\frac{(N-p)}{N+p^2-2p}},$$ we have  $$\|Du_{h,\lambda} -DU_{\lambda}\|_{L^p(\rr^N)}\mathrel{\lesssim} h^{\gamma(p,N)}.$$
 This allows us to use Corollary \ref{deficit.above} to obtain 
%
%
%
\begin{align*}
\label{}
  S_{h}&\leq \frac{\|Du_{h,\lambda}\|_{L^p(\rr^N)}}{\|u_{h,\lambda}\|_{L^{p^*}(\rr^N)}}\leq S + 
C(p,N)  \Big(\frac{\|Du_{h,\lambda} -DU_\lambda\|_{L^p(\rr^N)}}{\|Du_{h,\lambda}\|_{L^p(\rr^N)}}
\Big)^{\min\{2,p\}} .
\end{align*}
Since $\| DU_\lambda\|_{L^p(\rr^N)}=1$, we also get $\|Du_{h,\lambda}\|_{L^p(\rr^N)}\simeq 1$. Thus, we get the desired upper bound
\[
S_{h}-S\lesssim h^{\gamma(p,N)\min \{2,p\}}. 
\]
%
%
\textit{Step} II. \textit{The  lower bound.} Let $u_h\in V_h$ be a minimizer of $S_h$ with, for simplicity,  $\|Du_h\|_{L^{p}(\rr^N)}=1$.

From the previous upper bound and the estimate of the Sobolev deficit \eqref{deficit.below}, we have that there exists $U_{c_h,\lambda_h,x_h}=c_h U_{\lambda_h,x_h}\in \mathcal{M}$ such that
\begin{align*}
\label{}
h^{\gamma(p,N)\min\{2,p\}}&\gtrsim S_h-S
=\delta(u_h)\gtrsim \Big( \frac{\| Du_h-c_hDU_{\lambda_h,x_h}\|_{L^p(\rr^N)}}{\|Du_h\|_{L^{p}(\rr^N)}}\Big)^{\max\{2,p\}}.
\end{align*}

We now need to get a lower bound on the right-hand side of this inequality in terms of $h$. For this, we need to gather further information on $c_h$, $x_h $ and $\lambda_h.$ We therefore proceed in two steps.
\begin{itemize}
\item
In the following, we prove that $|c_h|\rightarrow 1$, $x_h\in B_h$, $\lambda_h\rightarrow\infty$ and $h\lambda_h^\frac{p}{N-p}\leq \Lambda$ for some positive constant $\Lambda$ as $h\rightarrow 0$. 

Since $\|Du_h\|_{L^{p}(\rr^N)}=\|DU_{\lambda_h,x_h}\|_{L^{p}(\rr^N)}=1$, it follows that 
\[|1-|c_h||\leq \|Du_h-c_hDU_{\lambda_h,x_h}\|_{L^{p}(\rr^N)}=o(1).\]
Choosing  $h$ small, we trivially have  $|c_h|^p>1/2$ and then
\begin{align*}
\label{}
o(1)&=\delta(u_h)^{\frac p{\max\{2,p\}}}\mathrel{\gtrsim}  \| Du_h-c_hDU_{\lambda_h,x_h}\|_{L^p(\rr^N)}^p\\
&=
|c_h|^p\int_{B_h^c} |DU_{\lambda_h,x_h} |^pdx+\int _{B_h}  |c_hDU_{\lambda_h,x_h}-Du_h|^pdx\\
&\gtrsim \int_{B^c} |DU_{\lambda_h,x_h} |^pdx+  \int _{B_h}  |DU_{\lambda_h,x_h}-c_h^{-1}Du_h|^pdx\\
&=\int_{B^c} |DU_{\lambda_h,x_h} |^pdx+I_{h}.
\end{align*}
Assuming $|x_h|\geq 1$, Remark \ref{x0mare} shows that for all    $\lambda>0$, the exterior of the unit ball  $B^c$ always contains a half-space passing through $x_h$ and then  
\[
o(1)\gtrsim \int_{B^c}|DU_{\lambda_h,x_h}|^pdx\geq \frac 12.
\]
Hence
 $|x_{h}|<1$.  A similar argument shows that $x_h\notin B\setminus B_h$. In this case, 
 \[
 \int_{B_h^c} |DU_{\lambda_{h},x_h} |^pdx\geq \frac 12.
  \]
 
{Let us now show that $\lambda_h\to\infty$. Otherwise, there exist $M>0$ and a subsequence along which $\lambda_h\le M$. Since $|x_h|<1$, we have $B\subset\{|x-x_h|<2\}$, and therefore
\[
  \begin{split}
  o(1)\gtrsim \int_{B^c}|DU_{\lambda_h,x_h}|^pdx&\geq \int_{|x-x_h|\geq 2}|DU_{\lambda_h,x_h}|^pdx\\
  &=\int_{|y|\geq 2\lambda_h^{\frac p{N-p}}}|DU|^pdy\geq \int_{|y|\geq 2M^{\frac p{N-p}}}|DU|^pdy>0,
\end{split}
\]
which gives a contradiction.}

Let us now prove that the sequence  $(\lambda_h^{\frac p{N-p}}h)_{h>0}$ is uniformly bounded. Suppose, by contradiction and up to a subsequence, that $\lambda_h^{\frac p{N-p}}h\rightarrow \infty$. We will show in this case that $$
I_{h}\gtrsim 1.$$

Let $T_{0h}\in \mathcal{T}_h$ be an element containing $x_h$.
If for some $\eps>0$ the ball $B_{\eps h}(x_h)$ is included in $T_{0h}$, then  $B_{{\lambda_{h}^{- p/(N-p)}}}(x_h)\subset B_{\eps h}(x_h)\subset T_{0h}$.
   Otherwise, we can  consider the intersection $T_{0h}\cap B_{\lambda_{h}^{- p/(N-p)}}(x_h)$, which still contains a cone $C$ centered at $x_h$ which occupies a fixed proportion (independent of $h$) of the ball $B_{\lambda_{h}^{- p/(N-p)}}(x_h)$ and $
  C\cap B_{\lambda_{h}^{- p/(N-p)}}(x_h)\subset T_{0h}.
 $

  Let us now explain in more detail this geometrical property and how to construct a  cone as above. For $x_0\in \rr^N$, $R>0$, $\alpha\in (0,1)$ and $\xi\in \mathbf{S}^{N-1}$, we define \textit{the cone}
  \[
  C_R(\alpha,\xi,x_0)=\{ x\in \rr^N, |x-x_0|<R, \, (x-x_0)\cdot \xi  \geq \alpha |x-x_0|\}.
  \]
  For the reference simplex $\widehat T$, there exist $R>0$ and $\alpha\in (0,1)$ such that 
  \[
  \forall \, \hat x\in \widehat T \ \exists\,  \hat\xi=\hat \xi(\hat x)\in \mathbf{S}^{N-1} \ \text{such that}\ C_{R}(\alpha, \hat \xi,\hat x )\subset \widehat T. 
  \]  
  We transfer this property to any $T\in \mathcal{T}_h$ by using the properties of the matrices $B_T$ in \eqref{bt} and  \eqref{regular}, \eqref{quasi-uniform}. In view of these properties, there exist $R_\sigma>0$ and $\alpha_\sigma\in (0,1)$ such that 
  \[
  \forall \, T\in \mathcal{T}_h,\  \forall   x\in   T, \ \exists\,  \xi=\xi(  x)\in \mathbf{S}^{N-1} \ \text{such that}\ C_{R_\sigma h}(\alpha_\sigma, \xi,  x )\subset   T. 
  \] 
  
%

 Let us now consider the case when $B_{{\lambda_{h}^{- p/(N-p)}}}(x_h)\subset T_{0h}$; otherwise we will apply the same argument to $C_{{\lambda_{h}^{- p/(N-p)}}}(\alpha_\sigma,\xi(x_h),  x_h)\subset T_{0h}$.
Since $Du_h$ is constant inside simplex $T_{0h}$, and satisfies  $c_h^{-1}Du_h=A_h$, we obtain a contradiction:
\begin{align*}
o(1)\gtrsim I_h&\geq 
\int_{|x-x_{h}|<\lambda_h^{-\frac p{N-p}}}| \lambda_{h}^{\frac {N}{N-p}}DU(\lambda_h ^{\frac p{N-p}}( x-x_h))-A_h|^pdx\\
&=\int_{|x|<1}|DU(x)-A_h\lambda_h ^{-\frac {N}{N-p}}|^pdx\geq \inf_{A\in \rr^N} \int_{|x|<1}|DU(x)-A|^pdx\geq C_0>0.
\end{align*}
Here $C_0>0$ since $DU$ is not constant in the unit ball. In the case of the cone, the integral is over $C_1(\alpha_\sigma,\xi(x_h),0)$; since $U$ is radial, the corresponding infimum is positive and independent of $\xi(x_h)$.

\item 
With the above considerations on $x_h$, $\lambda_h$ and $h\lambda_h^{\frac p{N-p}}$, we proceed to prove the desired lower bound. Let $\eta\in(0,1)$ be given by Lemma \ref{est.second.derivative} and fix an integer $K\geq 2/\eta$; $K$ depends only on $N$ and $p$, so that all the constants below are independent of $h$, $\lambda_h$ and $x_h$. Since $\lambda_h\to\infty$, choosing, if needed, a smaller value of $h$, we can assume that $h<1/(4(K+1))$ and
 \[
 h \lambda_h^{\frac p{N-p}}\leq \Lambda <\frac{\lambda_0^{p/(N-p)}}{8(K+1)}\leq {\frac{\lambda_h^{p/(N-p)}}{8(K+1)}},
  \]
where $\lambda_0>1$ is a fixed constant, chosen after $\Lambda$ so that the strict inequality above holds.
We use estimate \eqref{est.lambda.mare} to obtain \begin{equation}
\label{est.partial}
  \int_{B^c} |DU_{\lambda_h,x_h} |^p\,dx\geq {c_1}\lambda_h^{-\frac p{p-1}},
\end{equation}
for some positive constant $c_1$.

Let us now estimate the   term $I_{h}.$
Using that $Du_h$ is piecewise constant in each  simplex $T\in \mathcal{T}_h$, by Lemma \ref{lower.bound.triangle} and estimate \eqref{est.inferior.D2U} (with the direction $\xi=\xi_T$ chosen below), we obtain 
\begin{align*}
\label{}
  I_{h}&=\sum _{T\in \mathcal{T}_h}\int _T  |DU_{\lambda_h,x_h}-c_h^{-1}Du_h|^p\,dx\geq \sum _{T\in  \mathcal{T}_h}\min_{A_T\in \rr^N}\int _T  |DU_{\lambda_h,x_h}-A_T|^p\,dx\\
 &  \gtrsim\sum _{T\in  \mathcal{T}_h^{1}} h^{N+p}\lambda_{h}^{\frac{p(N+p)}{N-p}} \min_{x\in  {T}} a^p(\lambda_h^\frac{p}{N-p}|x-x_h|).
\end{align*}
We denote by   $ \mathcal{T}^1_h $ the elements $T\in \mathcal{T}_h$ which have a nonempty intersection with the set $\{|x-x_h|\geq (K+1)h\}$. 
%
Any element $T\in \mathcal{T}^1_h$
is included in $\{|x-x_h|\geq Kh\}$, and for any $x\in T\mathrel{\in}\mathcal{T}_h^1$, we have
\[
\max_{\overline x\in {T}}|\overline x-x_h|\leq |x-x_h|+h_T\leq |x-x_h|+h\leq 2|x-x_h|
\]
and
\[
|x-x_h|\leq \max_{\overline x\in  {T}}|\overline x-x_h|\leq 2\min_{\underline x\in  {T}}|\underline x-x_h|\leq 2|x-x_h|.
\]
{Moreover, for $T\in\mathcal{T}_h^1$ we choose $x_T\in T$ and $\xi_T\in\mathbf{S}^{N-1}$ with $\xi_T\cdot(x_T-x_h)=0$. Since $|x-x_T|\le h$ and $|x_T-x_h|\ge Kh$, for all $x\in T$ we have
\begin{align*}
	|\xi_T\cdot(x-x_h)|&=|\xi_T\cdot(x-x_T)|\leq h\le\frac1{K-1}|x-x_h|\le\eta|x-x_h|,
\end{align*}
where the last inequality holds since  $K\ge 2/\eta$ and $\eta<1$.
Hence \eqref{est.inferior.D2U} holds with the same direction $\xi=\xi_T$ at every point of $T$; this is the direction used above in Lemma \ref{lower.bound.triangle}. }
This implies that
\begin{align*}
\label{}
I_h\mathrel{\gtrsim} \sum _{T\in  \mathcal{T}^1_h} &h^{N+p}\lambda_h^{\frac{p(N+p)}{N-p}} \min_{x\in  {T}} a^p(\lambda_h^\frac{p}{N-p}|x-x_h|)
\gtrsim h^{p}\lambda_{h}^{\frac{p(N+p)}{N-p}}
 \sum _{T\in \mathcal{T}_h^1} \int_{T} a^p(\lambda_h^\frac{p}{N-p}|x-x_h|)\,dx\\
 &\gtrsim h^{p}\lambda_h^{\frac{p(N+p)}{N-p}} \int_{(K+1)h\leq|x-x_h|, |x|<1-h} a^p(\lambda_h^\frac{p}{N-p}|x-x_h|)\,dx.
\end{align*}

If $|x_h|<1/2$, then the set $\{(K+1)h\mathrel{\leq}|x-x_h|<1/4\}$ is included in $\omega_h=\{(K+1)h\mathrel{\leq}|x-x_h|, |x|<1-h\}\subset B_h$ since
$|x|\leq |x_h|+|x-x_h|<1/2+1/4<1-h$. This implies
\begin{align*}
\label{}
  I_h&\gtrsim h^{p}\lambda_h^{\frac{p(N+p)}{N-p}} \int_{(K+1)h\leq|x-x_h|<1/4} a^p(\lambda_h^\frac{p}{N-p}|x-x_h|)\,dx\\
  &\simeq 
h^{p}\lambda_h^{\frac{p^2}{N-p}} \int_{(K+1)h\lambda_h^\frac{p}{N-p}\mathrel{\leq}|y|<\lambda_h^\frac{p}{N-p}/4} a^p(|y|)\,dy\\
&\gtrsim h^{p}\lambda_h^{\frac{p^2}{N-p}}  \int_{(K+1)\Lambda <|x|<\lambda_0^{p/(N-p)}/4} a^p(|x|)\,dx.
\end{align*}
When $|x_h| \mathrel{\geq} 1/2$, it may happen that $x_h$ lies close to the boundary of $B_h$. Nevertheless, even in the worst-case scenario, we can always construct a truncated cone centered at $x_h$ that is entirely contained within the set $\omega_h$. Assume that $x_h=(x_{1h},0')$ with $x_{1h}\in [1/2,1)$. Thus,  
\[
\{ x=(x_1,x'), |x'-0'|\leq \alpha |x_{1h}-x_1|, (K+1)h\leq x_{1h}-x_1\leq \frac 12\}	\subset \omega_h
\]
since, in this case, $x_1\geq x_{1h}-\frac12\geq 0$, $|x-x_h|\geq x_{1h}-x_1\geq (K+1)h$ and, for $\alpha<1/2$,
\[
|x|\le x_1+\alpha(x_{1h}-x_1)=x_{1h}-(1-\alpha)(x_{1h}-x_1)\le x_{1h}-(1-\alpha)(K+1)h<1-h.
\]  Denoting
\[
\omega'_{h}=\{y=(y_1,y'), |y'|\leq \alpha |y_1|, (K+1)h\leq y_1\leq \frac 12\},	
\]
we obtain, after a change of variables, that
\begin{align*}
\label{}
  I_{h}&\gtrsim  h^{p}\lambda_h^{\frac{p^2}{N-p}} \int_{ y\in \lambda_h^\frac{p}{N-p} \omega'_{h}} a^p(|y|)dy\gtrsim  h^{p}\lambda_h^{\frac{p^2}{N-p}},
\end{align*}
where we have used that the set $\lambda_h^\frac{p}{N-p} \omega'_{h}$ contains a subset with positive measure, independent of the small parameter $h$,
\begin{align*}
  \{y=(y_1,y'), |y'|\leq &\alpha |y_1|, (K+1)\Lambda\leq y_1\leq \frac {\lambda_0^{\frac p{N-p}}}2\}\\
  &\subset 
\{y=(y_1,y'), |y'|\leq \alpha |y_1|, (K+1)h\lambda_h^\frac{p}{N-p}\leq y_1\leq \frac {\lambda_h^\frac{p}{N-p}}2\}	\subset \lambda_h^\frac{p}{N-p} \omega'_{h}.
\end{align*}
Putting together the estimates for $I_h$ and \eqref{est.partial}, we find that
\[
S_h-S\gtrsim ({\lambda_h^{-\frac 1{p-1}}}+ h\lambda_h^{\frac{p}{N-p}} )^{\max\{2,p\}}\gtrsim h^{\gamma(p,N)\max\{2,p\}}.
\]
which finishes the proof.
\end{itemize}
\end{proof}

\begin{proof}[Proof of Theorem \ref{main.2}] 
We proceed as above, but now considering the quasi-norms.

\noindent{ \textit{Step} I. \textit{The upper bound.}  }
The proof follows the same ideas as in Proposition \ref{main} by taking  $u_{h,\lambda}=I^h(U_{\lambda}-{U_{\lambda}}_{|_{\partial {B}}})\in V_h$. In view of estimate \eqref{est.u.liniar.1}, $\|Du_{h,\lambda}-DU_\lambda\|_p$ is small enough when $\lambda$ is large and $h\lambda ^{\frac p{N-p}}$ is small. This allows us to apply the upper bounds for the Sobolev deficit in Lemma \ref{deficit.above.new}.

For $p\geq 2$, the deficit estimates in Lemma \ref{deficit.above.new} give
\[
\delta (u_{h,\lambda})\lesssim \int _{\rr^N} (|Du_{h,\lambda}|+|DU_\lambda-Du_{h,\lambda}|)^{p-2}|DU_\lambda-Du_{h,\lambda}|^2dx
\]
	and estimate \eqref{est.u.liniar.2} leads to
	\[
	\delta (u_{h,\lambda})\lesssim \lambda^{-\frac p{p-1}}+(h\lambda^\frac{p}{N-p})^2.
	\]
	Taking  $\lambda$ such that  $ \lambda^{-\frac p{p-1}}=(h\lambda^\frac{p}{N-p})^2$, we find the desired estimate $\delta (u_{h,\lambda})\lesssim h^{\alpha(p,N)}$ with
\[\alpha(p,N)=\frac{2(N-p)}{N+p-2}.\]

	For $1<p<2$, we apply Lemma \ref{deficit.above.new} to get
	\begin{align*}
\delta(u_{h,\lambda})& \lesssim 	\int_{\rr^N}(|Du_{h,\lambda}|+ |D(U_\lambda-u_{h,\lambda})|)^{p-2}|D(U_\lambda-u_{h,\lambda})|^2 dx\\
&	\quad \quad +\Big( \int_{\rr^N}|DU_\lambda|^{p-1}  |D(U_\lambda-u_{h,\lambda})|\,dx\Big)^2.
\end{align*}
Using estimates \eqref{est.u.liniar.2} and \eqref{est.u.liniar.3}, we obtain 
\[
\delta(u_{h,\lambda})\lesssim \lambda^{-\frac p{p-1}}+(h\lambda^\frac{p}{N-p})^2.
\]
Choosing $\lambda$ appropriately, the proof of the upper bound is obtained as in the proof of Proposition \ref{main}.

\textit{Step} II. \textit{The lower bound.}  As in the proof of Proposition \ref{main}, let us take  $u_h\in V_h$ with $\|Du_h\|_{L^{p}(\rr^N)}=1$ and $\|u_h\|_{L^{p^*}(\rr^N)}=1/S_h$. From the upper bound estimate in Step I, we have
\[
h^{\alpha(p,N)}\gtrsim S_h-S =\delta(u_h).
\]
By Lemma \ref{deficit.below.quasi-norms}, there exists $U^h=U_{c_h,\lambda_h,x_h}$ such that   \eqref{deficit.below.2} holds:
\[
\delta(u_h)\gtrsim \int_{\rr^N}(|Du_{h }|+ |D(U^h-u_{h})|)^{p-2}|D(U^h-u_h)|^2 dx+\|D(U^h-u_h)\|^{\max\{2,p\}}_{L^p(\rr^N)}.
\]
The same analysis as in Step II of the proof of Proposition \ref{main} demonstrates that $|c_h|\simeq 1$, $x_h\in B_h$, $h\lambda_h^{\frac p{N-p}}\leq \Lambda$ and $\lambda_h\rightarrow\infty$. Also $\|D(U^h-u_h)\|_{L^p(\rr^N)}\leq 1$ for $h$ small and $\|DU^h\|_{L^p(\rr^N)}\simeq 1$.

For $p\geq 2$, the same analysis as in the proof of Proposition \ref{main} shows that there exists a set $\omega$ independent of $h$ such that 
\begin{align*}
  \delta(u_h)&\gtrsim  \int_{\rr^N}(|Du_{h }|+ |D(U^h-u_{h})|)^{p-2}|D(U^h-u_h)|^2 dx\\
  &\gtrsim \int_{\rr^N}|DU^h |^{p-2}|D(U^h-u_h)|^2 dx \\
  &=\int_{B_h^c}|DU^h |^{p }  dx+\int_{B_h}|DU^h |^{p-2}|D(U^h-u_h)|^2 dx\\
&\gtrsim \lambda_h^{-\frac p{p-1}} + (h\lambda_h^\frac{p}{N-p})^2\int _{\omega} |u_0'(|y|)|^{p-2} a^2(|y|)dy\gtrsim h^{\alpha(p,N)}.
\end{align*}
Here the integral over $B_h$ is estimated as in the proof of Proposition \ref{main}: on each $T\in\mathcal{T}_h^1$ the weight $|DU^h|^{p-2}$ is comparable to its value at any point of $T$ (recall that $|u_0'(r)|=|c_{N,p}|\,r\,a(r)$), Lemma \ref{lower.bound.triangle} is applied with exponent $2$ and direction $\xi_T$, and $\omega$ is the fixed annulus or truncated cone (away from the origin) used there, so that the last integral is finite and positive.

For $1<p<2$, let us first recall the following inequality for quasi-norms in \cite[Lemma 2.2]{MR1192966}:
\[
\int_{\Omega} (|Du|+|Dv|)^{p-2}|D(u-v)|^2\,dx\geq \frac{\|D(u-v)\|_{L^p(\Omega)}^2}{(\|Du\|_{L^p(\Omega)}+\|Dv\|_{L^p(\Omega)})^{
2-p}}.
\]
Using that $\|DU^h\|_{L^p(\rr^N)}\simeq 1$, we get 
\begin{align*}
  \delta(u_h)&\gtrsim  \int_{\rr^N}(|Du_{h }|+ |D(U^h-u_{h})|)^{p-2}|D(U^{h}-u_h)|^2 \,dx\\
  &\gtrsim \int_{B_h^c}|DU^h |^{p}dx+\int_{B_h}(|Du_{h }|+ |D(U^h-u_{h})|)^{p-2}|D(U^h-u_h)|^2 dx \\
&\gtrsim \lambda_h^{-\frac p{p-1}} +  \frac{\|D(U^h-u_h)\|_{L^p(B_h)}^2}{(\|DU^h\|_{L^p(B_h)}+\|D(U^{h}-u_h)\|_{L^p(B_h)})^{2-p}}\\
&\gtrsim \lambda_h^{-\frac p{p-1}} +  {\|D(U^h-u_h)\|_{L^p(B_h)}^2}.
\end{align*}
 From Step II of the proof of Proposition \ref{main} (the lower bound $I_h\gtrsim(h\lambda_h^{\frac p{N-p}})^p$ obtained there uses only $|c_h|\simeq1$, $x_h\in B_h$, $\lambda_h\to\infty$ and $h\lambda_h^{\frac p{N-p}}\leq\Lambda$, which hold here as well), we know that
$
 \|D(U^h-u_h)\|_{L^p(B_h)}\gtrsim h\lambda_h^{\frac p{N-p}}
 $; hence
 \[
   \delta(u_h)\gtrsim \lambda_h^{-\frac p{p-1}} + (h\lambda_h^{\frac p{N-p}})^2\gtrsim h^{\alpha(p,N)}.
 \]
 This completes the proof.
  \end{proof}
\section{Conclusions and perspectives}

In this paper, we have established sharp convergence rates for the $P1$ finite element approximations of the Sobolev constant. Our approach is inspired by existing techniques for analyzing the deficit function in Sobolev inequalities in the continuous setting but extends them through significant new developments. These advances are particularly crucial given that the Sobolev constant is not attained in bounded domains, necessitating a careful analysis of approximation rates for explicitly known minimizing sequences.

Moreover, achieving sharp convergence rates requires working with quasi-norms commonly used in the finite element analysis of the $p$-Laplacian rather than with the standard $W^{1, p}(\mathbb{R}^N)$ norms. Our methodology builds on prior results concerning finite element approximation rates for Laplacian eigenvalues, adapting them to the context of the Sobolev constant.

The techniques developed here can also be adapted to analyze $P1$ finite element approximations of other fundamental constants in the theory of PDEs and mathematical physics, such as the Hardy constant, which is  discussed in \cite{della2023finite,ignatDCDS2026}. This direction will be pursued in future work, as it requires a substantial refinement of the methodology introduced in this paper.

Finally, the questions explored here naturally extend to other settings, including the approximation of the Sobolev constant using different classes of finite element methods. However, addressing these extensions will require considerable additional work.

\appendix
\renewcommand{\sectionname}{}
\section{\texorpdfstring{Proofs of auxiliary results}{Proofs of auxiliary results}}\label{app}

\subsection{Piecewise linear approximations} In this section, we prove Lemma \ref{lower.bound.triangle}.
We first consider the one-dimensional case.
\begin{lemma}
	\label{1d-lower.bound} Let $-\infty<a<b<\infty$ and $u\in C^2([a,b])$. For any $p
	\in (1,\infty)$, there exists a positive constant $C(p)$ such that
	\begin{align*}
\label{est.1.lower.ineq}
  \int_a^b |u'(r)-A|^pdr\geq C(p){(b-a)^{p+1}}\inf_{r\in [a,b]}|u''(r)|^p, \ \forall \ A\in \rr.
\end{align*}
\end{lemma}

\begin{proof}
Consider the function $f(A)=\int_a^b |u'(r)-A|^{p}dr$. It is a convex function satisfying $\lim_{A\rightarrow\pm \infty}f(A)=\infty$. It attains its minimum at a point $A_0$ such that 
\[
\int_a^b |u'(r)-A_0|^{p-1}\sgn ((u'(r)-A_0))dr=0. 
\]
It follows that there exists $r_0\in (a,b)$ such that $u'(r_0)=A_0$.  	
 Using that 
 \[
 |u'(r)-u'(r_0)|\geq |r-r_0|\inf_{[a,b]}|u''(r)|,\]
  we get the desired estimate
  \[f(A)\geq f(A_0)\geq \inf_{[a,b]}|u''(r)|^p\int_a^b |r-r_0|^pdr\mathrel{\geq} C(p){(b-a)^{p+1}}\inf_{r\in [a,b]}|u''(r)|^p,\quad C(p)=\frac{1}{2^p(p+1)}.
  \]
This finishes the proof.
 \end{proof}

 \begin{proof}[Proof of Lemma \ref{lower.bound.triangle}]
 To fix the ideas, we prove it for the direction $x_1$.  For any $A_1\in \rr$, we have  
 		\[
 \int_T |\partial_{x_1}u-A_1|^pdx \geq  C(p)  \rho_T^{N+p} \min_{x\in  T} \,|\partial_{x_1x_1}u(x)|^p.
	\]
	Take a cube $C_{\rho_T}\subset B_{\rho_T}\subset T$ of side $L=\rho_T/\sqrt N$. To simplify the presentation, let us assume $C_{\rho_T}=[0, L]^N$. We write $x=(x_1,x')$ and apply the one-dimensional result in Lemma \ref{1d-lower.bound}:
	\begin{align*}
\int_T |\partial_{x_1}u-A_{1}|^pdx&\geq \int_{C_{\rho_T}} |\partial_{x_1}u-A_{1}|^pdx=\int_{[0,L]^{N-1}}\int_0^L  |\partial_{x_1}u-A_{1}|^pdx_1 dx'\\
&\geq C(p){L^{N+p}} \min_{x\in  T} \,|\partial_{x_1x_1}u(x)|^p\geq C(p) \rho_T^{N+p} \min_{x\in  T} \,|\partial_{x_1x_1}u(x)|^p.
\end{align*}

Let us now consider the general case. It is enough to prove that  for any $A\in \rr^N$ and $\xi\in  \mathbf{S}^{N-1}$,  it holds
\[
 \int_{B_{\rho_T}} |D u-A|^pdx \geq  C(p)  \rho_T^{N+p}\min _{x\in T}|\xi^T D^2u(x)\xi|^p.
\]
After a translation, we may assume that $B_{\rho_T}$ is centered at the origin, so that $R(B_{\rho_T})=B_{\rho_T}$ for every rotation $R$. Note also that the previous step only uses the values of $u$ in the cube $C_{\rho_T}\subset B_{\rho_T}$, so that it holds with $T$ replaced by $B_{\rho_T}$.
Consider a rotation $R$ of the unit sphere $ \mathbf{S}^{N-1}$ such that $Re_1=\xi $ and $v(y)=u(Ry)$. Using the previous step for $v$ and the fact that $Dv(y)=R^TDu(Ry)$, $D^2v(y)=R^T D^2u(Ry)R$, we get for any $A\in \rr^N$ that
\begin{align*}
 \min _{x\in B_{\rho_T}}|\xi^T D^2u(x)\xi|^p &= \min _{x\in  B_{\rho_T}}|e_1^T R ^TD^2u(x)R e_1|^p =
  \min _{x\in  B_{\rho_T}}|e_1^T  D^2v(x) e_1|^p\\
 & \leq \frac{1}{C(p)\rho_T^{N+p}}  \int_{B_{\rho_T}} |D v-A|^pdx=\frac{1}{C(p)\rho_T^{N+p}}  \int_{B_{\rho_T}} |R^{T}(D u)-A|^pdx\\
 &=\frac{1}{C(p)\rho_T^{N+p}}  \int_{B_{\rho_T}} |D u-RA|^pdx.
\end{align*}
Replacing $A$ by $R^TA$ gives the desired result.    
 \end{proof}

\subsection{An elementary inequality}
The following holds.
\begin{lemma}\label{inequalities}
Let  $q\in (1,\infty)$.  Then, there exist constants $A_q\geq 0$ and $B_q>0$ such that
\begin{equation}
\label{ineg.1}
\big||a+b|^q  -|a|^q-q|a|^{q-2}ab\big|  \leq A_q |a|^{q-2}|b|^2+ B_q |b|^q, \, \forall a,b\in \rr. 
\end{equation}
Furthermore, $A_q=0$ if $q\in (1,2]$.

On the other hand, for suitable $A_q$, $B_q$ and $C_q$, the following inequalities hold: 
\begin{equation}
\label{ineg.2}
|x+y|^q\leq |x|^q+q|x|^{q-2}x\cdot y +A_q |x|^{q-2}|y|^2+ B_q |y|^q, \, \forall x,y\in \rr^N 
\end{equation}
with $A_q=0$ when $q\in (1,2]$
and 
\begin{equation}
\label{ineg.3}
|x+y|^q\leq |x|^q+q|x|^{q-2}x\cdot y +C_q \frac{(|x|+|y|)^{q}} {|x|^2+|y|^2} |y|^2,  \, \forall x,y\in \rr^N.
\end{equation}
\end{lemma}
\begin{proof}Some cases are treated  in \cite[Lemma 3.2]{MR3896203}. For completeness, we give here a short proof. 

Let us begin with \eqref{ineg.1}. By homogeneity, we can reduce the proof to the following one: 
	\[
	\big||1+z|^q  -1 - q z\big|  \leq A_q |z|^2+ B_q |z|^q, \forall z\in \rr
.	\]
It is clear that for $|z|>1$, the inequality is true with $A_q=0$, independent of the value of $q$. When $|z|\mathrel{\leq} 1$, Taylor's expansion at $z=0$ and the continuity of $z\mapsto (|1+z|^q-1-qz)/z^2$ on $\{0<|z|\leq 1\}$ yield
\[
\big||1+z|^q-1-qz\big|\leq C_q|z|^2\leq C_q|z|^{\min \{2,q\}}, \quad |z|\leq 1.
\]
So, for $q<2$, one can choose $A_q=0$.

Inequality \eqref{ineg.2} holds trivially when $|x|=0$. Also
 one can consider only the case when $|x|=1$, and after a rotation, we can also assume $x=(1,0,\dots,0)$. It is then sufficient to prove that
 \[
 \Big( (1+y_1)^2 + |y'|^2 \Big)^{\frac q2}   \leq 1+q y_1 +A_q |y|^2+ B_q |y|^q
 \]
 for $y=(y_1,y') \in \rr\times \rr^{N-1}$.

 When $q\leq 2$, we easily have
 \[
  \Big( (1+y_1)^2 + |y'|^2 \Big)^{\frac q2}   \leq |1+y_1|^q + |y'|^q\leq 1+qy_1 +B_q|y_1|^q +|y'|^q\leq 1+q y_1 + \tilde B_q |y|^q.
 \]
 For $q>2$, we distinguish the cases  $q/2\leq  2$ and $q/2>2$. When $q/2\leq 2$,  in view of inequality \eqref{ineg.1},
 \begin{align*}
   \Big( (1+y_1)^2 + |y'|^2 \Big)^{\frac q2}& \leq |1+y_1|^q + \frac q2 |1+y_1|^{q-2}|y'|^2 +B_{q/2}|y'|^q   \\
   &\leq |1+y_1|^q  +  C_q(|y|^2 +|y|^q) \\
   &\leq 1+qy_1 +A_q|y_1|^2+B_q |y_1|^q + C_q(|y|^2 +|y|^q)\\
   &\leq 1+qy_1 + C_q'(|y|^2 +|y|^q).
\end{align*}
When $q/2>2$, using inequality \eqref{ineg.1} again, we have
 \begin{align*}
   \Big( (1+y_1)^2 + |y'|^2 \Big)^{\frac q2}& \leq |1+y_1|^q + \frac q2 |1+y_1|^{q-2}|y'|^2 +A_{q/2}|1+y_1|^{q-4}|y'|^4 +B_{q/2}|y'|^q   \\
   &\leq |1+y_1|^q  +  C_q(|y|^2 +|y|^q +|y|^4) \\
   &\leq 1+qy_1 +A_q|y_1|^2+B_q |y_1|^q + C_q(|y|^2 +|y|^q)\\
   &\leq 1+qy_1 + C_q'(|y|^2 +|y|^q).
\end{align*}

Let us now consider the last inequality \eqref{ineg.3}.
For $q\geq 2$, it is an immediate consequence of the second one. When 
$q\leq 2$, it remains to prove that
  \[
 \Big( (1+y_1)^2 + |y'|^2 \Big)^{\frac q2}   \leq 1+q y_1 +C_q \frac{(1+|y|)^{q}} {1+|y|^2} |y|^2. 
 \]
 In this case, we easily have
 \[
  \Big( (1+y_1)^2 + |y'|^2 \Big)^{\frac q2}   \leq |1+y_1|^q + |y'|^q\leq 1+qy_1 +B_q|y_1|^q +|y'|^q\leq 1+q y_1 + \tilde B_q |y|^q.
 \] For $|y|\geq 1$, we get the desired estimate since $|y|^q\lesssim {|y|^2 {(1+|y|)^{q}}/ (1+|y|^2)} $. When $|y|\leq 1$, we use that $(1+x)^\alpha\leq 1+\alpha x$ for all $0<\alpha\mathrel{\leq}1$ and $x>-1$:
 \begin{align*}
\label{}
   \Big( (1+y_1)^2 + |y'|^2 \Big)^{\frac q2}&  = (1+2y_1+y_1^2+|y'|^2)^{\frac q2}\leq 1+qy_1+\frac q2(y_1^2+|y'|^2)\\
   &\leq 1+qy_1+\frac q2 |y|^2\leq 1+qy_1+2\,\frac{(1+|y|)^{q}} {1+|y|^2} |y|^2,
\end{align*}
since $(1+|y|)^q/(1+|y|^2)\geq 1/2$ for $|y|\leq 1$.
 The proof is now complete.
\end{proof}

\subsection{Proof of Lemma \ref{deficit.above.new}} 
 
Let $u$ be such that $\|u\|_{L^{p^*}(\rr^N)}=1$ and
$$
\delta(u)=\|Du\|_{L^p(\rr^N)}-S(p,N)<S(p,N).
$$
	
	Then $S(p,N)\leq \|Du\|_{L^p(\rr^N)}\leq 2S(p,N)$, and, using $$|y^p-x^p|=p\,|(y-x)\xi^{p-1}|\geq |x|^{p-1}|y-x|, \quad {\text{for }0\leq x<y\text{ and some }\xi\in(x,y)},$$ we get
	\begin{equation}\label{quenosecomoseusa}
	\delta(u)=\|Du\|_{L^p(\rr^N)}-S(p,N)\leq \frac 1{(S(p,N))^{p-1}}(\|Du\|^p_{L^p(\rr^N)}-(S(p,N))^p).
	\end{equation}
It suffices, therefore, to estimate the right-hand side of the above inequality. 

Let $v\in \mathcal{M}$ and set  $\eps=\|D(u-v)\|_{L^p(\rr^N)}$. We write
\[
	u=v+\|Du -Dv\|_{L^p(\rr^N)}\frac{u-v}{\|Du -Dv\|_{L^p(\rr^N)}} :=v+\eps \varphi,
	\]
		where $\| D\varphi\|_{L^p(\rr^N)}=1$.  	
		 {When $\eps\leq \eps_0\|Du\|_{L^p(\rr^N)}$  with $\eps_0<1/2$, we have
		\[\frac{S(p,N)}2\leq \frac{\|Du\|_{L^p(\rr^N)}}2\leq \|Dv\|_{L^p(\rr^N)}\leq 2\|Du\|_{L^p(\rr^N)}\leq 4 S(p,N).
		\]
		Hence   $\|Dv\|_{L^p(\rr^N)}\simeq  \|Du\|_{L^p(\rr^N)}\simeq S(p,N)$. 
		Since $\|u\|_{L^{p^*}(\rr^N)}=1$ and
		\[\|u-v\|_{L^{p^*}(\rr^N)}\leq {S(p,N)^{-1}}\|Du-Dv\|_{L^p(\rr^N)}\leq {2\eps_0},
		\]
		 for $\eps_0$ small enough, we get that
		\[
		\frac 12\leq   \| v\|_{L^{p^*}(\rr^N)}\leq 2.
		\]
		}
		
The two cases \( 1 < p < 2 \) and \( 2 \leq p < N \) are treated differently.

\noindent\textit{Case} I. $2\leq p<N$.		
	Inequality \eqref{ineg.3} shows that  for any $q>1$ and $x,y\in \rr^N$, the following holds:
	\begin{equation}
\label{ineq.q-2}
|x+y|^q\leq |x|^q+q|x|^{q-2}x\cdot y +C_q {(|x|+|y|)^{q-2}} |y|^2.
	\end{equation}		
	Thus, we obtain
\begin{align*}
  \|Du\|_{L^p(\rr^N)}^p&=\int_{\rr^N}|D(v+\eps\varphi)|^p dx\\
  &\leq \int_{\rr^N}|D v |^pdx+\eps p|Dv|^{p-2}Dv \cdot D\varphi+C_p\eps^2(|Dv|+\eps |D\varphi|)^{p-2}|D\varphi|^2 dx.
\end{align*}
We  use that for $p^*>1$, inequality $|1+x|^{p^*}\geq 1+p^* x $ holds for all $x\in \rr$  to get
\begin{align*}
\| u\|_{L^{p*}(\rr^N)}^p&\geq \Big(\int_{\rr^N} (|v|^{p^*}+\eps p^* |v|^{p^*-2}v\varphi) \,dx \Big)^{\frac p{p^*}}=
\| v\|_{L^{p^*}}^p  \Big(1+\eps p^* \frac{\int_{\rr^N} |v|^{p^*-2}v\varphi \,dx}{\| v\|_{L^{p^*}}^{p^*}}\Big)^{\frac p{p^*}}.
\end{align*}
Since {$\|v\|_{L^{p^*}(\rr^N)}\simeq 1$} and
\[\Big|\int_{\rr^N} |v|^{p^*-2}v\varphi \,dx\Big|\leq \|v\|^{p^*-1}_{L^{p^*}(\rr^N)} \|\varphi\|_{L^{p^*}(\rr^N)}\lesssim 1\] we obtain 
\[
\frac{|\int_{\rr^N} |v|^{p^*-2}v\varphi\,dx|}{\| v\|_{L^{p^*}}^{p^*}}\lesssim 1.
\]
{Taylor's expansion at $x=0$ of $(1+x)^\alpha=1+\binom \alpha 1 x\,+\,\binom \alpha 2 x^2 +\dots $  shows that
\[(1+\eps x)^{\frac p{p^*}}\geq 1+ \eps \frac{px}{p^*}-C\eps ^2 x^2
\] 
holds for all $|x|\leq M$ and $0<\eps\leq 1/(2M)$, where $M>0$ is fixed and $C=C(p,N)$.} We apply it with $x=p^*\int_{\rr^N}|v|^{p^*-2}v\varphi\,dx/\|v\|^{p^*}_{L^{p^*}}$, which is bounded by the previous estimate, to get
\begin{align*}
\| u\|_{L^{p*}}^p&\geq \| v\|_{L^{p^*}}^p  \Big(1+\eps p \frac{\int_{\rr^N} |v|^{p^*-2}v\varphi\,dx}{\| v\|_{L^{p^*}}^{p^*}}-C\eps^2\Big( \frac{\int_{\rr^N} |v|^{p^*-2}v\varphi\,dx}{\| v\|_{L^{p^*}}^{p^*}}\Big)^2\Big).
\end{align*}
Putting together the above two estimates and using that 
\begin{equation}
\label{id.1}
  \int_{\rr^N}|Dv|^pdx=(S(p,N))^p \|v\|_{L^{p^*}(\rr^N)}^p,
\end{equation}
\begin{equation}
\label{id.2}
  \int_{\rr^N}|Dv|^{p-2}Dv\cdot  D\varphi dx=(S(p,N))^p \|v\|_{L^{p^*}(\rr^N)}^{p-p^*}\int_{\rr^N} |v|^{p^*-2}v\varphi dx,
\end{equation}
we get
\begin{align*}
\label{}
  \delta(u)&\lesssim \|Du\|_{L^p(\rr^N)}^p-(S(p,N))^p \| u\|_{L^{p*}}^p\\
  &\mathrel{\lesssim}\eps^2\int_{\rr^N}(|Dv|+\eps |D\varphi|)^{p-2}|D\varphi|^2 dx+\eps^2\Big(\int_{\rr^N} |v|^{p^*-2}v\varphi\,dx\Big )^2 \|v\|_{L^{p^*}(\rr^N)}^{p-2p^*}.
\end{align*}
H\"older's inequality and \cite[Lemma 3.3]{MR4484209} give us
\begin{align*}
  \Big(\int_{\rr^N} |v|^{p^*-2}v\varphi dx\Big)^2\leq \Big( \int_{\rr^N} |v|^{p^*-2} \varphi^2\,dx \Big)  \Big( \int_{\rr^N} |v|^{p^*} \,dx \Big)\lesssim  \|v\|_{L^{p^*}(\rr^N)}^{p^*}
\int_{\rr^N} |Dv|^{p-2}|D\varphi|^2\,dx.
\end{align*}
This shows that
\begin{align*}
    \delta(u)& \lesssim \eps^2\int_{\rr^N}(|Dv|+\eps |D\varphi|)^{p-2}|D\varphi|^2 dx +\eps^2 {\int_{\rr^N} |Dv|^{p-2}|D\varphi|^2\,dx }\\
    &= \int_{\rr^N}(|Dv|+  |D(u-v)|)^{p-2}|D(u-v)|^2 dx + {\int_{\rr^N} |Dv|^{p-2}|D(u-v)|^2\,dx}.
\end{align*}
This is exactly the estimate in Remark \ref{estimate.general.p}. Note that the above computation does not use $p\geq 2$: inequality \eqref{ineq.q-2} holds for every $q>1$ and \cite[Lemma 3.3]{MR4484209} holds for every $1<p<\infty$; hence \eqref{general.not.well} holds for every $1<p<N$. Since $p\geq 2$, we immediately obtain \eqref{eq.deficit.above}.  

\textit{Case} II. $1<p<2$. 
 In this setting,\[
\| u\|_{L^{p*}}^{p^*}= \int_{\rr^N} |v+\eps \varphi|^{p^*}\,dx\geq \int_{\rr^N} |v|^{p^*}\,dx+\eps p^* \int_{\rr^N} |v|^{p^*-2}v\varphi 
\,dx\]
and by inequality \eqref{ineg.3},
\begin{align*}
\label{}
  \|Du&\|_{L^p(\rr^N)}^{p^*}=\Big(\int_{\rr^N}|D(v+\eps\varphi)|^p\,dx  \Big)^{\frac{p^*}p}\\
  &\leq \Big(\int_{\rr^N}|D v |^p +\eps p|Dv|^{p-2}Dv\mathbin{\cdot} D\varphi+C_p\eps^2(|Dv|+\eps |D\varphi|)^{p-2}|D\varphi|^2  \,dx\Big)^{\frac{p^*}p}\\
  &=\|Dv\|_{L^p(\rr^N)}^{p^*}\Big(1+\frac{\eps p\int_{\rr^N}|Dv|^{p-2}Dv\mathbin{\cdot} D\varphi{\,dx}+C_p\int_{\rr^N}\eps^2(|Dv|+\eps |D\varphi|)^{p-2}|D\varphi|^2\,dx }{\|Dv\|_{L^p(\rr^N)}^p}\Big)^{\frac{p^*}p}.
\end{align*}
For small $x$, we have $$|1+x|^{\frac{p^*}p}\leq 1+\frac {p^*}px+Cx^2$$ and this gives us
\begin{align*}
    \|Du\|_{L^p(\rr^N)}^{p^*}\leq&\|Dv\|_{L^p(\rr^N)}^{p^*}+\eps p^* \|Dv\|_{L^p(\rr^N)}^{p^*-p}\int_{\rr^N}|Dv|^{p-2}Dv\mathbin{\cdot} D\varphi\,dx\\
    &+\eps^2 \frac{C_pp^*}{p} \|Dv\|_{L^p(\rr^N)}^{p^*-p}\int_{\rr^N}(|Dv|+\eps |D\varphi|)^{p-2}|D\varphi|^2 dx\\
    &+C \|Dv\|_{L^p(\rr^N)}^{p^*-2p} \Big(\eps p\int_{\rr^N}|Dv|^{p-2}Dv\mathbin{\cdot} D\varphi \,dx\\
    &\qquad\qquad\qquad\qquad +C_p\int_{\rr^N}\eps^2(|Dv|+\eps |D\varphi|)^{p-2}|D\varphi|^2 dx\Big)^2.
\end{align*}
For $1<p<2$, we trivially have
\[
\eps^2 \int_{\rr^N}(|Dv|+\eps |D\varphi|)^{p-2}|D\varphi|^2 dx\leq \eps^{p}\int_{\rr^N} |D\varphi|^p dx=\eps^p.
\]
Identities \eqref{id.1} and \eqref{id.2} imply that
\begin{align*}
    \delta(u)&\lesssim \|Du\|_{L^p(\rr^N)}^{p^*}-(S(p,N))^{p^*} \| u\|_{L^{p*}}^{p^*}\\
    &\lesssim \eps^2  \int_{\rr^N}(|Dv|+\eps |D\varphi|)^{p-2}|D\varphi|^2 dx+\Big(\eps \int_{\rr^N}|Dv|^{p-2}Dv\mathbin{\cdot} D\varphi \,dx\Big)^2\\
    &\mathrel{\leq} \int_{\rr^N}(|Dv|+ |D(u-v)|)^{p-2}|D(u-v)|^2 dx+\Big( \int_{\rr^N}|Dv|^{p-1}  |D(u-v)|\,dx \Big)^2.
\end{align*}
This finishes the proof. 
\subsection{Proof of Corollary \ref{deficit.above}} 

	Consider $u\in \dot W^{1,p}$ with $\| Du\|_{L^p(\rr^N)}=1$ such that 
	$d(u,\mathcal{M})=\eps$ with $\eps$ a small positive parameter  so that  there exists $v=v_u\in \mathcal{M}$ such that $d(u,\mathcal{M})=\| D(u-v)\|_{L^p(\rr^N)}=\eps$. Since $v\in \mathcal{M}$ we have  $\|u\|_{L^{p^*}(\rr^N)}\simeq \|v\|_{L^{p^*}(\rr^N)} $ and 
	$\delta(u)$ is such that Lemma \ref{deficit.above.new} applies.
	For $p\geq 2$ and $\eps<1$, we have
	\begin{align*}
    \delta(u)&\lesssim	\int_{\rr^N}(|Du|+|Dv|)^{p-2}|D(u-v)|^{2}\,dx \lesssim \int_{\rr^N}(|D(u-v)|+|Du|)^{p-2}|D(u-v)|^{2}\,dx\\
    &\lesssim	  \|D(u-v)\|_{L^p(\rr^N)}^{p} + \|Du\|_{L^p(\rr^N)}^{p-2}\|D(u-v)\|_{L^p(\rr^N)}^{2}   \\
    &\lesssim \|D(u-v)\|_{L^p(\rr^N)}^{2}.
\end{align*}
A similar argument works for  $1<p<2$ since $\|Dv\|_{L^p(\rr^N)}\leq 2$ for small $\eps$:
	\begin{align*}
  \delta(u)&\lesssim 	\int_{\rr^N}(|Dv|+ |D(u-v)|)^{p-2}|D(u-v)|^2 dx +\Big( \int_{\rr^N}|Dv|^{p-1}  |D(u-v)|dx\Big)^2\\
  &\lesssim  	\int_{\rr^N} |D(u-v)|^p dx+ \left(\int_{\rr^N}|Dv|^{p}dx\right)^{\frac 2{p'}}\left(	\int_{\rr^N} |D(u-v)|^p dx\right)^{\frac 2p}\\
  &\lesssim \|D(u-v)\|_{L^p(\rr^N)}^p + \|D(u-v)\|_{L^p(\rr^N)}^2\lesssim \|D(u-v)\|_{L^p(\rr^N)}^p.
\end{align*}
Taking the infimum over all such $v\in\mathcal{M}$, we obtain \eqref{upperboundx}. This finishes the proof.

\subsection{Proof of Lemma
	\ref{unif.est.U.lambda}}
	
	Let us assume, without loss of generality, that $x_0=(x_{01},0,\dots,0)$ with $|x_{01}|<1$. Using the explicit form of $U_{\lambda,x_0}(x)$,  and denoting 
	\[
	A_\lambda=\{(y_1,y')\in \rr\times \rr^{N-1},\ |(y_1+\lambda^{\frac p{N-p}}x_{01},y')|\leq \lambda^{\frac{p}{N-p}}\},
	\]
	we obtain, by an elementary change of variables,
	\begin{align*}
 \int_{|x|<1} |DU_{\lambda,x_0}|^pdx=  \int_{A_\lambda}|DU(y)|^pdy=C_N \int_{A_\lambda}\frac{|y|^{\frac p{p-1}}dy}{(1+|y|^\frac p{p-1})^N}.
\end{align*}
Observe that, since $|x_{01}|<1$, 
$A_\lambda\subset \{y\in \rr^N, \ |y|\leq 2\lambda^{\frac p{N-p}}\}$. Then the upper bound in \eqref{est.lambda.mic} holds for all $\lambda>0$: 
\begin{align*}
   \int_{|x|<1} |DU_{\lambda,x_0}|^pdx&\lesssim \int_{|y|\leq 2\lambda^{\frac p{N-p}}} \frac{|y|^\frac p{p-1}dy}{(1+|y|^\frac p{p-1})^N}\lesssim \int_0^{2\lambda^{\frac p{N-p}}}\frac{r^{N-1+\frac p{p-1}}dr}{(1+r^\frac p{p-1})^N}\\
   &\leq\int_0^{2\lambda^{\frac p{N-p}}}r^{N-1+\frac p{p-1}}dr\simeq\lambda^{\frac{p}{N-p}(N+\frac p{p-1})}.
\end{align*}
For the lower bound in \eqref{est.lambda.mic}, observe that for $|x_{01}|\leq 1/2$, there is always a ball centered at the origin included in $A_\lambda$,
$\{|y|<\lambda^{\frac p{N-p}}/2\}\subset A_\lambda$. Then
\[
  \int_{|x|<1} |DU_{\lambda,x_0}|^pdx\gtrsim \int_{|y|<\lambda^{\frac p{N-p}}/2}\frac{|y|^\frac p{p-1}dy}{(1+|y|^\frac p{p-1})^N}\gtrsim \lambda^{\frac{p}{N-p}(N+\frac p{p-1})}, \ \forall \ \lambda<1. 
\] 
Let us now consider the case when $1/2<|x_{01}|<1$ and choose the particular case $1/2<x_{01}<1$. Observe that \[
y\in A_\lambda  \iff |y'|^2+y_1^2 +2y_1\lambda^{\frac p{N-p}}x_{01}\leq \lambda^{\frac {2p}{N-p}}(1-x_{01}^2)\]
and then, for $\lambda$ small enough,
\[
B_\lambda:=\{(y_1,y')\in \rr\times \rr^{N-1}, |y'|^2+y_1^2 +y_1\lambda^{\frac p{N-p}} \leq 0, \ y_1\leq 0\}\subset A_\lambda\subset \{|y|<1\}.
\]
Moreover, $B_\lambda =\lambda^{\frac p{N-p}}B_1$ and then 
\[
\int_{B_\lambda}\frac{|y|^\frac p{p-1}dy}{(1+|y|^\frac p{p-1})^N}\geq  \frac 1{2^N}\int_{B_\lambda}{|y|^\frac p{p-1}dy} \simeq  \lambda^{\frac{p}{N-p}(N+\frac p{p-1})} \int _{B_1}|y|^{\frac p{p-1}}dy.
\]
The proof of estimate \eqref{est.lambda.mic} is complete. 

Let us now consider the inequalities in \eqref{est.lambda.mare}. They are reduced to estimating the minimizers outside the unit ball 
\[
\int_{|x|>1} |DU_{\lambda,x_0}|^pdx=C_N \int_{A_\lambda^c}\frac{|y|^{\frac p{p-1}}dy}{(1+|y|^{\frac p{p-1}})^N}.\]
Using that  for $\lambda>1$, $\{y\in \rr^N, |y|\geq 2\lambda^{\frac {p}{N-p}}\}\subset A_\lambda^c$, we have
\[
\int_{A_\lambda^c}\frac{|y|^{\frac p{p-1}}dy}{(1+|y|^{\frac p{p-1}})^N}\geq 
\int_{|y|\geq 2\lambda^{\frac {p}{N-p}}} \frac{|y|^{\frac p{p-1}}dy}{(1+|y|^{\frac p{p-1}})^N}\geq {2^{-N}}\int_{|y|\geq 2\lambda^{\frac {p}{N-p}}} {|y|^{-\frac {p(N-1)}{p-1}}}dy\simeq \lambda^{-\frac {p}{p-1}}.
\]

%
Also observe that $A_\lambda=\lambda^{\frac{p}{N-p}}A_1$ and any $y\in A_1^c$, that is, $|y+(x_{01},0)|\geq 1$, satisfies $|y|\geq 1-|x_{01}|$.
Hence 
\begin{align*}
\label{}
  \int_{A_\lambda^c}\frac{|y|^\frac{p}{p-1}dy}{(1+|y|^\frac{p}{p-1})^N}
&= \lambda^{\frac{p}{N-p}(N+\frac p{p-1})} \int_{A_1^c}\frac{|y|^\frac{p}{p-1}dy}{(1+\lambda^{\frac {p^2}{(N-p)(p-1)}}|y|^\frac{p}{p-1})^N}\\
&\leq  \lambda^{\frac{p}{N-p}(N+\frac p{p-1})} \int_{|y|>1-|x_{01}|}\frac{|y|^\frac{p}{p-1}dy}{(1+\lambda^{\frac {p^2}{(N-p)(p-1)}}|y|^\frac{p}{p-1})^N}\\
&\leq {\lambda^{-\frac p{p-1}}} \int_{|y|\geq 1-|x_{01}|}{|y|^{-\frac p{p-1}(N-1)}}dy\mathrel{\simeq}\frac{{\lambda^{-\frac p{p-1}}}}{(1-|x_{01}|)^{\frac{N-p}{p-1}}}.
\end{align*}
This proves \eqref{est.lambda.mare} and finishes the proof.

\subsection{Proof of Lemma \ref{est.second.derivative}}\label{app.D2U}

 Using a translation and scaling argument, it suffices to consider the case $x_0=0$ and $\lambda=1$.
	For the radially symmetric function $U(x)=u_0(|x|)$, we have
\[
\partial_{x_ix_j}U(x)=u_0''(r)\frac{x_ix_j}{r^2}+u_0'(r)\frac{\delta_{ij}r^2-x_ix_j}{r^3}
\]
and the eigenvalues of the Hessian matrix are $\lambda_k(D^2U(x))\in \{u_0''(r), \frac{u_0'(r)}r\}$. 

Denoting $c_{N,p}=-k_0(N-p)/(p-1)$, explicit computations show that
\[
u_0'(r)=c_{N,p}r^{\frac 1{p-1}}(1+r^{\frac p{p-1}})^{-\frac Np},
\]
\[u_0''(r)=\frac{c_{N,p}}{p-1}(1+r^{\frac p{p-1}})^{-\frac Np-1}   (1-(N-1)r^{\frac{p}{p-1}})r^{\frac{2-p}{p-1 }},
\]
and
\begin{equation}
	\label{laplacian}
 \Delta U(x)=\frac{c_{N,p}}{p-1}r^{\frac {2-p}{p-1}}\Big(  1+r^{\frac p{p-1}}\Big)^{-\frac Np-1} \Big[ (p-2)(N-1) r^{\frac p{p-1}}+ 1+(p-1)(N-1)\Big].
\end{equation}
Moreover,
 for any $1\leq i, j\leq {N}$, we have
\[
|\partial_{x_ix_j}U(x)|\lesssim |u_0''(r)|+\frac{|u_0'(r)|}r\lesssim r^{\frac {2-p}{p-1}} \Big(  1+r^{\frac p{p-1}}\Big)^{-\frac Np }.\]

It implies that under the assumption $1<p<N$, $D^2U\in L^p(\rr^N)$. Indeed, 
\begin{align*}
\int_{\rr^N}|D^2U(x)|^pdx
   \lesssim\int_0^\infty r^{N-1} r^{\frac{(2-p)p}{p-1}}  \Big(  1+r^{\frac p{p-1}}\Big)^{-N  }dr<\infty
\end{align*}
since 
\[
N+\frac{p(2-p)}{p-1}-\frac {Np}{p-1}<0 \Leftrightarrow {(p-1)^2+N-1>0}
\]
and
\[
N+\frac{p(2-p)}{p-1}>0 \Leftrightarrow p\in (\frac{N+2-\sqrt{N^2+4}}2,\frac{N+2+\sqrt{N^2+4}}2 ) \supseteq (1,N).
\]
 A scaling argument proves \eqref{est.superior.D2U} and \eqref{est.bila.unitate}.

{Let us now prove \eqref{est.inferior.D2U}. By translation and scaling, it suffices to consider $x_0=0$ and $\lambda=1$ (the condition $|\xi\cdot(x-x_0)|\le\eta|x-x_0|$ is invariant under these transformations). For $x\neq0$, let $r=|x|$, $e=x/r$ and $t=r^{\frac p{p-1}}$. Since 
\[D^2U(x)=u_0''(r)\,e\otimes e+\frac{u_0'(r)}r(I-e\otimes e),\]
 and
 \[\frac{u_0'(r)}r=c_{N,p}\,a(r), \quad \frac{r\,u_0''(r)}{u_0'(r)}=\frac{1-(N-1)t}{(p-1)(1+t)},\] we obtain, for $\xi\in\mathbf{S}^{N-1}$,
\[
\xi^TD^2U(x)\xi=c_{N,p}\,a(r)\big[1+b(t)(\xi\cdot e)^2\big],
\]
where \[ b(t)=\frac{(2-p)-(N+p-2)t}{(p-1)(1+t)}.\]
The function $b$ is monotone on $[0,\infty)$, so that $|b(t)|\le B:=\max\{|2-p|,N+p-2\}/(p-1)$. Choosing $\eta\in(0,1)$ with $B\eta^2\le1/2$ (so that $\eta$ depends only on $N$ and $p$), we get $|\xi^TD^2U(x)\xi|\ge {|c_{N,p}|}\,a(r)/2$ whenever $|\xi\cdot e|\le\eta$, which is \eqref{est.inferior.D2U} with $A_{N,p}=|c_{N,p}|/2$.

Note that a lower bound with a fixed direction cannot hold in a whole neighborhood of $x_0$: the Hessian degenerates on the sphere $\{r^{\frac p{p-1}}=1/(N-1)\}$, where $u_0''=0$, and at $x=0$ when $p<2$.}

\subsection{Proof of Lemma \ref{est.distance.vh}}\label{est.distanc}

 For simplicity, we write $U_\lambda$ instead of $U_{\lambda,0}$. 
We split all the integrals over $B_h$ and $B_h^c$. Since $u_h$ vanishes outside $B_h$, in $B_h^c$ the error $D(U_\lambda-u_h)$ equals $DU_\lambda$, so that the contribution of $B_h^c$ to each of the integrals in \eqref{est.u.liniar.1}--\eqref{est.u.liniar.3} is bounded by $\max\{1,2^{p-2}\}\int_{B_h^c}|DU_\lambda|^pdx$.  In view of the way the polyhedral domain $B_h$ is constructed, $B_h^c\subset  \{ |x|>1/2\}$ (for $h$ small).
 The same arguments as in the proof of \eqref{est.lambda.mare} lead to
\begin{align*}
\int_{B_h^c} |DU_{\lambda }|^pdx\leq     \int_{|x|>1/2}|DU_{\lambda }|^pdx\lesssim \lambda^{-\frac p{p-1}}.
\end{align*}
It remains to estimate the integrals in $B_h$. We proceed in several steps.

\noindent \textit{Step} I. \textit{Proof of} \eqref{est.u.liniar.1}. 
Let us denote
\[
I_h=   \int_{B_h}|DU_{\lambda}-Du_h|^pdx.
  \]
When $p>N/2$, classical estimates for the linear interpolator \cite[Theorem 4.4.20, p. 108]{MR2373954} and the result in \eqref{est.bila.unitate} show that
\[
I_h\mathrel{ {\lesssim}} h^p \|D^2 U_\lambda\|_{L^p(B_h)}^p\lesssim (h\lambda^{\frac p{N-p}})^p,
\]
which finishes the proof of \eqref{est.u.liniar.1}. 
 
 Let us now consider the case $1<p\leq N/2$,  in which we will make use of the estimate provided in Lemma \ref{est.c2}.
 We estimate differently the terms in the simplices  $T$ which are outside the ball $\{|x|<2h\}$. 
  If $T\in \mathcal{T}_h$ with $T\subset \{|x|>2h\}$, then $h_T\leq h<\frac 12\inf_{y\in T}|y|$ and therefore, for all $x\in T$, $\frac 12|x|\leq \inf_{y\in T}|y|\leq |x|\leq \sup_{y\in T}|y|\leq 2\inf_{y\in T}|y|$. In this case, estimate \eqref{est.superior.D2U} gives us 
 \[
 |T|\|D^2U_\lambda\|_{L^\infty(T)}^p\lesssim   |T| \sup_{x\in T} \lambda^{\frac{p(N+p)}{N-p}} a^{p}(\lambda^{\frac p{N-p}}|x|)\lesssim \lambda^{\frac{p(N+p)}{N-p}} \int_T  a^p(\lambda^{\frac p{N-p}}{|x|})dx.
 \]
Using  Lemma \ref{est.c2}, it follows that
\begin{align*}
 I_{1h}:=& \sum _{T\in \mathcal{T}_h, T\subset\{|x|>2h\}} \int_T |DU_{\lambda}-Du_h|^pdx \lesssim h^p
    \sum _{T\in \mathcal{T}_h, T\subset\{|x|>2h\}} |T|\|D^2U_\lambda\|_{L^\infty(T)}^p\\
    &\lesssim h^p \lambda^{\frac{p(N+p)}{N-p}} \int_{2h<|x|<1} a^p(\lambda^{\frac p{N-p}}{|x|})dx  \lesssim h^p \lambda^{\frac{p^2}{N-p}}  \int_0^\infty a^p(  r)r^{N-1}dr\lesssim h^p\lambda^{\frac{p^2}{N-p}}.
\end{align*}
The simplices \( T \in \mathcal{T}_h \) that are not entirely contained outside the ball \( \{|x| < 2h\} \) are included within the larger ball \( \{|x| < 3h\} \). We denote by $I_{2h}$ the part of the integral corresponding to these simplices. 
We recall the following estimate for the linear interpolator, see \eqref{est.fem.c3.loc}:
\[
\|DI^hu\|_{L^p(T)}\lesssim 
	h_T^{-1+\frac Np}\|u\|_{L^\infty(T)},\   p \geq 1.
\]
From the construction of $B_h$, all boundary nodes of \(B_h\) lie on \(\partial B\). Also since \(U_\lambda\) is radial, 
the nodal values of \(U_\lambda\) on \(\partial B_h\) are all equal to
\(C_\lambda:=U_\lambda|_{\partial B}\). Hence the nodal interpolant
\(I_h(U_\lambda-C_\lambda)\) belongs to \(V_h\), and
\(Du_h=D I_h(U_\lambda-C_\lambda)=D I_hU_\lambda\) elementwise. Thus,
\begin{align*}
 I_{2h}&\mathrel{\lesssim}   \sum _{T\in \mathcal{T}_h, T\subset\{|x|<3h\}} \int_T (|DU_{\lambda}|^p+|Du_h|^p)dx \\
 &\lesssim  \int_{|x|<3h}  |DU_{\lambda}|^pdx +h^{N-p}\|U_\lambda\|_{L^\infty(|x|<3h)}^p\\
 &\lesssim  \int_0^{3h\lambda ^{\frac p{N-p}}} |u_0'(s)|^ps^{N-1}ds+h^{N-p}\lambda^p \|U\|_{L^\infty(\rr^N)}^p \\
 &\lesssim (h \lambda^\frac{p}{N-p})^{N+\frac p{p-1}}+(h \lambda^\frac{p}{N-p})^{N-p}\lesssim  (h \lambda^\frac{p}{N-p})^{N-p}.
\end{align*}
Summing with the estimate for $I_{1h}$ and using that $1<p\leq N/2$, we get the desired result
\[
I_h\lesssim ( h   \lambda^{\frac p{N-p}})^p+
(h \lambda^\frac{p}{N-p})^{N-p}\lesssim ( h   \lambda^{\frac p{N-p}})^p. 
\] 
%
 
\noindent \textit{Step} II. \textit{Proof of} \eqref{est.u.liniar.2}. Let us first consider the case $p\geq 2$.  For $p>N/2$, we use the error estimates for the quasi-norms in  \cite[Theorem 3.1]{MR2135783} to get
\begin{align*}
  I_{h}&:=  \int_{B_h} (|DU_{\lambda} |+|DU_{\lambda} -Du_h|)^{p-2} |DU_{\lambda} -Du_h|^2dx\\
  & \lesssim  h^2 \int_{B_h} (|DU_\lambda|+h|D^2U_\lambda|)^{p-2}|D^2U_\lambda|^2dx\\
   & \lesssim  h^2 \int_{B_h} |DU_\lambda|^{p-2}|D^2U_\lambda|^2dx+h^{p}\int_{B_h} |D^2U_\lambda|^{p}dx \\
 &\mathrel{ {\lesssim}}  ( h   \lambda^{\frac p{N-p}})^2 \int_0^\infty |u_0'(r)|^{p-2} a^2(r)r^{N-1}dr+h^p \lambda^{\frac{p^2}{N-p}}  \int_0^\infty a^p(  r)r^{N-1}dr\\
 &\lesssim  ( h   \lambda^{\frac p{N-p}})^2+( h   \lambda^{\frac p{N-p}})^p\lesssim  ( h   \lambda^{\frac p{N-p}})^2.
\end{align*}
Let us consider the case $2\leq p\leq N/2$. We split $I_h=I_{1h}+I_{2h}$ as in Step I.
For the first term, we use arguments similar to those in Step I:
\begin{align*}
 I_{1h}:&= \sum _{T\in \mathcal{T}_h, T\subset\{|x|>2h\}} \int_T (|DU_{\lambda} |+|DU_{\lambda} -Du_h|)^{p-2} |DU_{\lambda} -Du_h|^2dx\\
&\mathrel{ {\lesssim}}   \sum _{T\in \mathcal{T}_h, T\subset\{|x|>2h\}}  \int_T (|DU_{\lambda} |^{p-2} |DU_{\lambda} -Du_h|^2+|DU_{\lambda} -Du_h|^{p})dx  \\
  &\mathrel{ {\lesssim}}
    \sum _{T\in \mathcal{T}_h, T\subset\{|x|>2h\}} \|DU_\lambda\|_{L^\infty(T)}^{p-2} \int_{T}|DU_{\lambda}-Du_h|^2dx+h^p\lambda^{\frac{p^2}{N-p}}\\
    &\lesssim  h^p\lambda^{\frac{p^2}{N-p}}+ h^2 \int_{|x|>2h} |DU_\lambda|^{p-2}|D^2U_\lambda|^2dx\\
    &\mathrel{ {\lesssim}} ( h   \lambda^{\frac p{N-p}})^p+  ( h   \lambda^{\frac p{N-p}})^2 \int_0^\infty |u_0'(r)|^{p-2} a^2(r)r^{N-1}dr\lesssim  ( h   \lambda^{\frac p{N-p}})^2.
\end{align*}
For the term $I_{2h}$, we proceed as in Step I since $N-p\geq p\geq 2$:
 \begin{align*}
\label{}
  I_{2h} & \leq \sum _{T\in \mathcal{T}_h, T\subset\{|x|<3h\}} \int_T( |DU_{\lambda} |+|DU_{\lambda} -Du_h|)^{p-2} |DU_{\lambda} -Du_h|^2dx\\
  &\mathrel{ {\lesssim}}   \sum _{T\in \mathcal{T}_h, T\subset\{|x|<3h\}} \int_T (|DU_{\lambda}|^p+|Du_h|^p)dx \lesssim  (h\lambda^\frac{p}{N-p})^{N-p}\leq (h\lambda^\frac{p}{N-p})^{2}.
\end{align*}

Let us now consider the case $1<p< 2$:
\begin{align*}
\label{}
 I_{1h}&=  \sum _{T\in \mathcal{T}_h, T\subset\{|x|>2h\}} \int_T  (|DU_{\lambda} |+|DU_{\lambda} -Du_h|)^{p-2} |DU_{\lambda} -Du_h|^2dx\\
 &\mathrel{ {\lesssim}}  \sum _{T\in \mathcal{T}_h, T\subset\{|x|>2h\}}h^2 \|D^2U_{\lambda} \|_{L^\infty(T)}^2\int_T  |DU_{\lambda} |^{p-2}dx \\
 &\mathrel{ {\lesssim}} h^2 \int_{|x|>2h}   (\lambda^{\frac N{N-p}})^{p-2} |u_0'(\lambda^{\frac p{N-p}}|x|)|^{p-2}
 \lambda^{\frac{2(N+p)}{N-p}}|a(\lambda^{\frac p{N-p}}|x|)|^2dx\\
 &\mathrel{ {\lesssim}}  ( h   \lambda^{\frac p{N-p}})^2 \int_0^\infty |u_0'(r)|^{p-2} a^2(r)r^{N-1}dr\lesssim  ( h   \lambda^{\frac p{N-p}})^2.
 \end{align*}
We use that  for $p\leq 2$ the function $a$ is uniformly bounded near the origin to get
 \begin{align*}
 I_{2h}&\leq    \sum _{T\in \mathcal{T}_h, T\subset\{|x|<3h\}} \int_T  (|DU_{\lambda} |+|DU_{\lambda} -Du_h|)^{p-2} |DU_{\lambda} -Du_h|^2dx\\
 &\leq   \sum _{T\in \mathcal{T}_h, T\subset\{|x|<3h\}} \int_T   |DU_{\lambda} -Du_h|^pdx\mathrel{\lesssim} h^p
   \sum _{T\in \mathcal{T}_h, T\subset\{|x|<3h\}} |T|\|D^2U_\lambda\|_{L^\infty(T)}^p\\
   &\mathrel{ {\lesssim}}  h^{p+N}\lambda^{\frac{(N+p)p}{N-p}} a^p(\lambda^{\frac p{N-p}}{h)} \lesssim  ( h   \lambda^{\frac p{N-p}})^{N+p}\leq ( h   \lambda^{\frac p{N-p}})^{2}. 
 \end{align*}
\begin{remark} \label{singularremark}At this point, it is important to observe that even if the following inequality holds   \begin{align*}
 I_{2h}&\leq    \sum _{T\in \mathcal{T}_h, T\subset\{|x|<3h\}} \int_T  (|DU_{\lambda} |+|DU_{\lambda} -Du_h|)^{p-2} |DU_{\lambda} -Du_h|^2dx\\
 &\leq   \sum _{T\in \mathcal{T}_h, T\subset\{|x|<3h\}} \int_T   |DU_\lambda|^{p-2}|DU_{\lambda} -Du_h|^2dx,
  \end{align*}
  it is useless in practice for the whole range $1<p<2$, since
the last integral may diverge over the simplex \( T \) containing the origin. This is so as \( |DU_\lambda|^{p-2} \) fails to be integrable in a neighborhood of the origin when $1<p\leq\frac{N+2}{N+1}$, in contrast to \( |DU_\lambda|^p \), which remains integrable. This is why the estimate in \eqref{general.not.well} had to be improved to the one in \eqref{eq.deficit.above.2}.
\end{remark}
 
\textit{Step} III. \textit{Proof of} \eqref{est.u.liniar.3}. With notations similar to those in the previous steps,
\begin{align*}
 I_{1h}:= &\sum _{T\in \mathcal{T}_h, T\subset\{|x|>2h\}} \int_T |DU_{\lambda} |^{p-1} |DU_{\lambda} -Du_h| dx\\
 &\mathrel{ {\lesssim}} h
    \sum _{T\in \mathcal{T}_h, T\subset\{|x|>2h\}}  \|D^2U_\lambda\|_{L^\infty(T)} \int_T |DU_{\lambda} |^{p-1}dx\\
    &\lesssim h (\lambda ^{1+\frac p{N-p}})^{p-1}
  \lambda^{\frac{(N+p)}{N-p}} \int_{2h<|x|<1} |u_0'(\lambda^{\frac p{N-p}}{|x|}) |^{p-1}a(\lambda^{\frac p{N-p}}{|x|})dx\\
    &\mathrel{ {\lesssim}}  h \lambda ^{\frac {N(p-1)}{N-p}}
  \lambda^{\frac{(N+p-pN)}{N-p}}     \int_0^\infty |u_0'(r)|^{p-1}a(  r)r^{N-1}dr\lesssim h \lambda^{\frac{p}{N-p}}.
\end{align*}
For the second term $I_{2h}$, we proceed as before, but we cannot bound   $\|D^2 U_\lambda\|_{L^\infty(T)}$ by an elementwise integral as above. Instead, since we are in the ball $\{|x|<3h\}$ and $p<2$, we can compare it, up to a constant, with the values at $|x|=h$ and use that $a(r)\lesssim r^{\frac{2-p}{p-1}}$ for $0<r<1$:
 \begin{align*}
 I_{2h}&\mathrel{ {\lesssim}}  h \sum _{T\in \mathcal{T}_h, T\subset\{|x|<3h\}} \int_T |DU_{\lambda} |^{p-1}  \|D^2U_\lambda\|_{L^\infty(T)}\,dx\\
 &\mathrel{ {\lesssim}} h \lambda^{\frac{(N+p)}{N-p}} a(\lambda^{\frac p{N-p}}{h)}  \lambda ^{\frac {N(p-1)}{N-p}}\int_{|x|<3h} |u_0'(\lambda^{\frac p{N-p}}{|x|}) |^{p-1}dx\\
 &\mathrel{ {\lesssim}} h\lambda^{\frac{(p+pN)}{N-p}} (\lambda^{\frac p{N-p}}{h)}^{\frac{2-p}{p-1}} \lambda^{\frac{-pN}{N-p}} \int_{0}^{3} |u_0'(r) |^{p-1}r^{N-1}dr\\
 &\mathrel{ {\lesssim}} (h \lambda ^{\frac {p}{N-p}})^{\frac{1}{p-1}}\lesssim h \lambda^{\frac{p}{N-p}}.
 \end{align*}
The proof is now complete.

\subsection*{Acknowledgments}
The authors thank Endre S\"uli for his valuable suggestion of employing the quasi-norms in \cite{MR1192966}  to sharpen the preliminary convergence results in Proposition \ref{main} to Theorem \ref{main.2}, on the occasion of the workshop ``Modelling, PDE analysis and computational mathematics in materials science",
Prague, September 2024. The authors also thank the anonymous referees for their valuable suggestions.
This work began while L.~I.~Ignat was visiting Friedrich-Alexander-Universit\"at Erlangen-N\"urnberg as part of the FAU Visiting Professor Program.
The authors thank Andreea Dima and Drago\c s Manea for their valuable comments on earlier versions of the manuscript and Sergiu Moroianu for insightful discussions concerning several geometric aspects of the paper.

\subsection*{Funding}
L.~I.~Ignat was partially supported by a grant of the Ministry of Research, Innovation, and Digitization, CCCDI-UEFISCDI, project no.~ROSUA-2024-0001, within PNCDI IV.
E. Zuazua was partially supported by the European Research Council (ERC) under the European Union's Horizon Europe research and innovation programme (grant agreement No.\ 101096251-CoDeFeL); by the Alexander von Humboldt Professorship program;  the AFOSR 24IOE027 project; the SURE-AI Norwegian Centre for Sustainable, Risk-Averse, and Ethical AI grant 357482, Research Council of Norway; and by the Grant PID2023-146872OB-I00-DyCMaMod of MICIU (Spain).

Both authors were partially supported  by the COST Actions CA24122 (Multiscale Stochastics, Patterns, and Analysis of Combinatorial Environments) and CA24136 (Interactions
between Control Theory and Machine Learning).

\printbibliography

\end{document}